\documentclass[11pt]{amsart}

\usepackage[utf8]{inputenc}
\usepackage[T1]{fontenc}

\usepackage[top=25mm, bottom=25mm, left=30mm, right=30mm]{geometry}  

\usepackage{amsmath,amssymb, amscd, color}
\usepackage{graphicx}
\usepackage{hyperref}
\usepackage{comment}

\usepackage[style=alphabetic,natbib=true,giveninits=true,url=false]{biblatex}
\DeclareFieldFormat{title}{\mkbibquote{#1}}
\AtEveryBibitem{\clearfield{issn}}
\AtEveryCitekey{\clearfield{issn}}

\numberwithin{equation}{section}

\newtheorem{theorem}{Theorem}[section]

\newtheorem{lemma}[theorem]{Lemma}
\newtheorem{corollary}[theorem]{Corollary}
\newtheorem{proposition}[theorem]{Proposition}

\newtheorem*{claim*}{Claim}

\newtheorem{definition}[theorem]{Definition}

\newtheorem{remark}[theorem]{Remark}

\usepackage{colonequals}
\newcommand{\defeq}{\colonequals}
\usepackage{booktabs}
\DeclareMathOperator{\vol}{Vol}
\DeclareMathOperator{\hor}{Hor}
\DeclareMathOperator{\ver}{Ver}
\DeclareMathOperator{\up}{Up}
\DeclareMathOperator{\low}{Low}
\DeclareMathOperator{\cyc}{Cyc}
\DeclareMathOperator{\Ext}{Ext}

\DeclareMathOperator{\SL}{SL}
\DeclareMathOperator{\interior}{int}
\DeclareMathOperator{\Area}{Area}

\renewcommand{\setminus}{-}
\renewcommand{\geq}{\geqslant}
\renewcommand{\leq}{\leqslant}
\newcommand{\cross}{\times}
\newcommand{\from}{\colon}
\newcommand{\join}{\vee}
\newcommand{\meet}{\wedge}
\newcommand{\st}{\mid}
\newcommand{\calA}{\mathcal{A}}
\newcommand{\calB}{\mathcal{B}}
\newcommand{\calC}{\mathcal{C}}
\newcommand{\calD}{\mathcal{D}}
\newcommand{\calF}{\mathcal{F}}
\newcommand{\calJ}{\mathcal{J}}
\newcommand{\calP}{\mathcal{P}}
\newcommand{\calQ}{\mathcal{Q}}
\newcommand{\calT}{\mathcal{T}}

\newcommand{\calV}{\mathcal{V}}
\newcommand{\calW}{\mathcal{W}}
\newcommand{\CC}{\mathbb{C}}
\newcommand{\NN}{\mathbb{N}}
\newcommand{\PP}{\mathbb{P}}

\newcommand{\RR}{\mathbb{R}}
\newcommand{\ZZ}{\mathbb{Z}}
\newcommand{\cover}[1]{\widetilde{#1}}
\newcommand{\bdy}{\partial}

\newcommand{\teichmuller}{Teichm\"{u}ller}
\newcommand{\Teichmuller}{TEICHM\"{U}LLER}

\makeatletter
 \let\c@theorem=\c@subsection{}
 \let\c@conjecture=\c@subsection{}
 \let\c@lemma=\c@subsection{}
 \let\c@proposition=\c@subsection{}
 \let\c@claim=\c@subsection{}
 \let\c@question=\c@subsection{}
 \let\c@criterion=\c@subsection{}
 \let\c@vfconj=\c@subsection{}
 \let\c@definition=\c@subsection{}
 \let\c@notation=\c@subsection{}
 \let\c@remark=\c@subsection{}
 \let\c@example=\c@subsection{}
 \let\c@equation=\c@subsection{}
 \let\c@figure=\c@subsection{}
 \let\c@wrapfigure=\c@subsection{}

\makeatother

\author[Bell]{Mark Bell}
\address{\hskip-\parindent{}
  Independent\\
  UK}
\email{mcbell@illinois.edu}

\author[Delecroix]{Vincent Delecroix}
\address{\hskip-\parindent{}
 Max Planck Institute for Mathematics\\
 Vivatsgasse 7\\
 53111 Bonn}
\email{vincent.delecroix@u-bordeaux.fr}

\author[Gadre]{Vaibhav Gadre}
\address{\hskip-\parindent{}
  School of Mathematics and Statistics\\
  University of Glasgow\\
  Glasgow, G128QQ UK}
\email{vaibhav.gadre@glasgow.ac.uk}

\author[Guti\'{e}rrez-Romo]{Rodolfo Gutiérrez-Romo}
\address{\hskip-\parindent{}
        Institut de Math\'{e}matiques de Jussieu--Paris Rive Gauche\\
        Universit\'{e} Paris 7\\
        Paris, France}
\email{rodolfo.gutierrez@imj-prg.fr}

\author[Schleimer]{Saul Schleimer}
\address{\hskip-\parindent{}
    	Department of Mathematics\\
        University of Warwick\\
        Coventry, CV47AL UK}
\email{s.schleimer@warwick.ac.uk}

\thanks{This work is in the public domain.}

\title{Coding \Teichmuller{} flow using veering triangulations}

\begin{document}
\begin{abstract}
We develop the theory of veering triangulations on oriented surfaces adapted to moduli spaces of half-translation surfaces.
We use veering triangulations to give a coding of the \teichmuller{} flow on connected components of strata of quadratic differentials.
We prove that this coding, given by a countable shift, has an approximate product structure and a roof function with exponential tails. This makes it conducive to the study of the dynamics of \teichmuller{} flow.
\end{abstract}

\maketitle

\section{Introduction}
\label{s:intro}

The $\SL(2,\RR)$ action on moduli spaces of meromorphic quadratic differentials with at most simple poles on finite-type Riemann surfaces has been a centrepiece in the field of \teichmuller{} dynamics.
The diagonal part of the action is called the \emph{\teichmuller{} flow}.
The dynamical properties of the flow have been a subject of tremendous interest.

We introduce a coding of the \teichmuller{} flow using veering triangulations of the underlying surfaces adapted to the given stratum of quadratic differentials.
Our main theorem is:

\begin{theorem}\label{t:main}
The coding by veering triangulations exhibits the \teichmuller{} flow as a suspension flow over a countable shift with approximate product structure and a roof function with exponential tails.
\end{theorem}

We give the precise definitions of the technical terms in the statement of the theorem later in the paper.
We view this paper as foundational material for work in progress.

The coding of the \teichmuller{} flow by interval exchange maps~\cite{Rau1,Rau2,Vee} or more generally linear involutions (non-classical interval exchanges)~\cite{Boi-Lan} is commonly used in applications.
A serious limitation of the interval exchange coding is that it requires a choice of a horizontal separatrix.
As a result, it actually encodes the \teichmuller{} flow on a finite cover of the stratum; the degree of the cover is linear in the complexity of the surface.
The coding by veering triangulations that we develop here has the advantage that it does not require taking a lift.
Moreover our main theorem, Theorem~\ref{t:main}, shows the veering triangulation coding has some of the same features as the interval exchange coding.
This will be important for subsequent applications of the coding.

\subsection{Notes and references}

Triangulations of surfaces have been used broadly in low dimensional geometry, topology and dynamics, and specifically in \teichmuller{} theory.
The use of flip sequences of triangulations (alternatively known as diagonal changes) to study flat geometry problems along a \teichmuller{} trajectory can be found in the work of Delecroix--Ulcigrai~\cite{Del-Ulc} for the hyper-elliptic strata of Abelian differentials and Ferenczi~\cite{Fer} for general Abelian strata.
The notion of a veering triangulation on an orientable surface is due to Agol~\cite{Ago} who used these triangulations to study the topology of mapping tori of pseudo-Anosov mapping classes.

This paper uses parameter spaces of veering triangulations to give a coding of the \teichmuller{} flow.
Some aspects of this connection have been independently discovered by Frenkel~\cite{Fre} who used it to analyse certain quantitative aspects of the (semi)-hyperbolicity of the flow.
The coding by train tracks developed by Hamenst\"{a}dt~\cite{HamCode} is similar in spirit to the one given here.  

\subsection{Acknowledgements}

We would like to thank Ian Frankel and Carlos Matheus for illuminating discussions related to this work.
The third and fourth named authors would like to specifically thank Carlos Matheus for extensive discussions on the estimates derived in Section \ref{s:estimates}.
Several of the ideas in Section~\ref{s:veering} first appeared in the joint work of the second author and Corinna Ulcigrai~\cite{Del-Ulc, Del-Ulc-2}.

\section{Veering triangulations}
\label{s:veering}

Let $S$ be a connected oriented surface of finite type with genus $g$ and $p$ punctures.
The space of marked conformal structures on $S$ is called the \teichmuller{} space of $S$.
The mapping class group is the group of orientation preserving diffeomorphisms of $S$ modulo isotopy.
It acts on \teichmuller{} space by changing the marking.
The action is properly discontinuous.
The quotient is the moduli space of Riemann surfaces homeomorphic to $S$.
A standard reference for this material is \cite{Far-Mar}.

\subsection{Strata of quadratic differentials:}

The cotangent bundle $\calQ (S)$ to (\teichmuller{}) moduli space is the space of (marked) holomorphic quadratic differentials on $S$ with simple poles at punctures.
The space $\calQ (S)$ is naturally stratified by the orders of the zeroes.
There is a standard generalisation of the strata motivated by the inclusion of regular points.
We proceed as follows.

Suppose that $S$ has no punctures. 
Let $Z \subset S$ be a finite non-empty set. 
We call the points $z \in Z$ \emph{marked} points.
A \emph{numerical datum} is a function $\kappa\colon Z \to \{ -1, 0, 1, 2, \ldots \}$ so that $\sum \kappa(z) = 4g - 4$.
We denote by $\calQ(\kappa)$ the stratum of meromorphic quadratic differentials $q$ on $S$ that have, for $z \in Z$, 
\begin{itemize}
\item $\kappa(z) \geq 1$ if and only if $z$ is a zero of that order, 
\item $\kappa(z) = 0$ if and only if $z$ is a regular point, and
\item $\kappa(z) = -1$ if and only if $z$ is a simple pole.
\end{itemize}
Note that not all strata are non-empty and not all strata are connected; further invariants such as hyper-ellipticity and spin are needed to distinguish components. 
For the complete classification, see~\cite{Kon-Zor},~\cite{Lan} and~\cite{Che-Moe}.

\subsection{Topological veering triangulations}
\label{sec:topological-veering-triangulations}

Again, suppose that $S$ is a closed connected oriented surface of genus $g$.
Suppose that $Z \subset S$ is finite and non-empty. 
Suppose that $\tau$ is a (topological) \emph{triangulation} of $(S, Z)$.
So $\tau$ gives a collection of \emph{edges} $E(\tau)$ and a collection of \emph{triangles} $F(\tau)$.  
The vertices of $\tau$ are exactly the points of $Z$. 
The usual Euler characteristic argument proves that $\tau$ has $4g - 4 + 2 |Z|$ triangles and $6g - 6 + 3|Z|$ edges.

Suppose that $e$ is an edge of $\tau$.  
Suppose that $T$ and $T'$ are the faces of $\tau$ adjacent to $e$.  
Let $Q = Q(e, \tau)$ be the union of the interior of $e$ with the interiors of $T$ and $T'$.  
If $T$ and $T'$ are distinct then $Q$ is topologically a quadrilateral.  
In this case we may \emph{flip} at $e$ to transform $\tau$ into a new triangulation $\tau'$.  
That is, we remove $e$ from $\tau$ and replace it by the other diagonal $e'$ of $Q$. 
Also, we remove $T$ and $T'$ from $\tau$ and replace them by the triangles $Q - e'$.  

To record how edges in $\tau'$ relate to edges of $\tau$ we introduce \emph{labels}.  
Let $\calA$ be an alphabet of size $|E(\tau)|$.  
A \emph{labelling} of $\tau$ is a bijection $\pi \from \calA \to E(\tau)$.
A \emph{labelled triangulation} is such a pair $(\tau, \pi)$. 
If $\pi$ is a labelling of $\tau$, and if $\tau$ flips to $\tau'$ by removing $e$ and adding $e'$, then we define the \emph{induced labelling} $\pi' \from \calA \to E(\tau')$ by
\[
(\pi')^{-1}(e') = \pi^{-1}(e) \qquad \text{and} \qquad (\pi')^{-1}(d) = \pi^{-1}(d) \text{ for $d \in E(\tau) - \{e\}$}.
\]
Then $(\tau', \pi')$ is the resulting labelled triangulation. 

We generalise triangulations to \emph{coloured triangulations}.
Here every edge receives one of two colours: \emph{red} or \emph{blue}.
Although a coloured triangulation is combinatorial, the colours are meant to represent slopes in a half-translation structure: red for positive slope and blue for negative slope.

\begin{definition}
\label{d:veering}
A coloured triangulation $\tau$ is \emph{veering} if it has no monochromatic triangles and no monochromatic vertices.
\end{definition}

Suppose that $\tau$ is a veering triangulation of $(S, Z)$.
About any marked point $z \in Z$ we see a circularly-ordered list of edge ends.
These give us a circularly-ordered list of colours.
Let $\kappa'(z)$ be the number of adjacent pairs of the form ``red then blue'' in the list.
That is, $\kappa'(z)$ counts the number of transitions in the list from red to blue.
Now define $\kappa \from Z \to \{-1, 0, 1, \ldots\}$ by $\kappa(z) = \kappa'(z) - 2$.
Then $\kappa$ is a numerical datum and we say that $\tau$ is combinatorially in the stratum determined by the data $(S, Z, \kappa)$.

\begin{lemma}
\label{l:veering-quads}
Suppose that $e$ is an edge of a veering triangulation $\tau$.  
Then $Q = Q(e, \tau)$ is a quadrilateral.
\end{lemma}

\begin{proof}
Suppose that $T = T'$.  
Then either $Q$ is a M\"obius band, contradicting the orientability of $S$, or there is a degree one vertex inside of $Q$.  
But this vertex is thus monochromatic, contradicting the veering hypothesis.
\end{proof}

Suppose that $\tau$ is a veering triangulation of $(S, Z)$.  
Fix an edge $e$ and equip it with an orientation.  
Let $T$ and $T'$ be the (distinct) triangles to the left and right of $e$.  
Suppose that the edges of $\partial Q(e, \tau)$ are, in anti-clockwise order, $a$, $b$, $c$, and $d$.  
Thus the edges of $T$ are (again in anti-clockwise order) $a$, $b$, and $e$. 
We say that $e$ is \emph{(topologically) forward flippable} if $a$ and $c$ are blue while $b$ and $d$ are red.
Similarly, $e$ is \emph{(topologically) backward flippable} if $a$ and $c$ are red and $b$ and $d$ are blue.
See Figure~\ref{fig:flippable_edges}.

\begin{figure}[ht]
\centering
\includegraphics{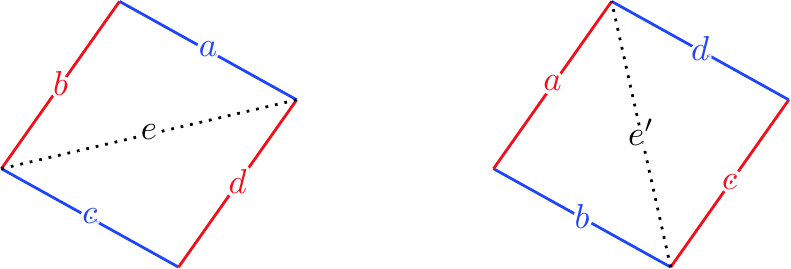}
\caption{A forward flippable edge $e$ and backwards flippable edge $e'$.}
\label{fig:flippable_edges}
\end{figure}

Note that if we flip a (forward or backward) flippable edge $e$ to the other diagonal $e'$ in $Q(e,\tau)$, the new triangulation $\tau'$ is again veering.
Moreover, the edge $e$ is forward flippable edge if and only if its flip, $e'$, is backward flippable.

\subsection{Veering polytope}
\label{s:veering-polytope}

Suppose that $(\tau, \pi)$ is a labelled veering triangulation of $(S, Z)$. 

To each label $a$ we associate a positive real variable $x^{(\tau, \pi)}_a$ which we think of as a width.
For each triangle $T$ we have a triangle equality of widths as follows.
Suppose that the labels of the edges of $T$ are $a$, $b$, and $c$ ordered anti-clockwise.  
Suppose that 
\begin{itemize}
\item $a$ and $b$ are blue and $c$ is red or
\item $c$ and $a$ are red and $b$ is blue.
\end{itemize}
Then we have
\begin{equation}
\label{eq:triangle-equality}
x_a = x_b + x_c.
\end{equation}
Replacing $x$ by $y$ we similarly define the heights and their triangle equalities. 

The \emph{width polytope} $\hor_\tau$ is the polytope in $\RR^{\calA}_{\geq 0}$ defined by the $x$-relations~(\ref{eq:triangle-equality}). 
The \emph{height polytope} $\ver_\tau$ is the polytope in $\RR^{\calA}_{\geq 0}$ defined by the analogous $y$-relations. 
The \emph{veering polytope} is $C_\tau = \hor_\tau \cross \ver_\tau$.  
In Section~\ref{s:balanced} we will use this to construct half-translation surfaces from topological veering triangulations. 

\subsection{Half-translation structure:}

A meromorphic quadratic differential $q$ on $S$ is equivalent to a particular type of singular flat structure on $S$.
Away from its zeroes and poles $q$ defines charts to $\CC$ with transition functions that are \emph{half-translations} $z \to \pm z+ c$.
By the Removable Singularity theorem, the structure extends over the singular points.
The cone angle at a point $z \in Z$ is equal to $\pi \kappa'(z)$.
The $\SL(2,\RR)$ action on $\RR^2 = \CC$ preserves the transition functions.
So it descends to an action on $\calQ(\kappa)$.
The diagonal part of this action is called the \emph{\teichmuller{} flow}.
We denote the flow by $\phi_t(q)$.

The notion of horizontal and vertical inside a chart is preserved under the transition functions.
Hence, $q$ gives a well-defined notion of a vertical and a horizontal at every point of $S$.
These local structures give horizontal and vertical measured foliations on $S$.
The \teichmuller{} flow stretches the horizontal foliation and shrinks the vertical foliation by the same factor.

A straight line segment $\gamma$ on the half-translation surface $(S,q)$ is called a \emph{saddle connection} if 
\begin{itemize}
\item the interior of $\gamma$ is embedded in $S - Z$ and
\item the endpoints of $\gamma$ lie in $Z$.
\end{itemize}
Note that saddle connections may end at regular points.

By locally choosing a root of $q$, we can associate to a saddle connection $\gamma$ a complex number  
\[
z(\gamma) = \int\limits_{\gamma} \sqrt{q}.
\]
Since we chose a root, $z(\gamma)$ is only defined up to sign. 
Nonetheless the \emph{slope} 
\[
m(\gamma) = \frac{\text{Im}\, z(\gamma)}{\text{Re} \, z(\gamma)}
\]
is well-defined. 
When $\text{Re} \, z(\gamma) = 0$ then we take $m(\gamma) = \infty$.
Finally we define $x(\gamma) = \left| \text{Re} \, z(\gamma) \right|$ and $y(\gamma) = \left| \text{Im} \, z(\gamma) \right|$ to be the \emph{width} and \emph{height}, respectively, of $\gamma$. 
 
\begin{definition}
A half-translation surface $(S,q)$ is \emph{horizontally Keane} if it has no horizontal saddle connections.
Similarly, it is \emph{vertically Keane} if it has no vertical saddle connections.
We will say that that $(S,q)$ is \emph{Keane} if it is both horizontally and vertically Keane.
\end{definition}


\subsection{Straight triangulations:}

A triangulation on a half-translation surface is \emph{straight} if the edges are saddle connections. 
It is well-known that straight triangulations exist. 
For example, see \cite[Proof of Lemma 1.1]{Mas-Smi}. 
We were unable to find a reference including a simple proof, so we give one here.


\begin{lemma}\label{l:straight}
Every half-translation surface $(S,q)$ admits a straight triangulation.
\end{lemma}

\begin{proof}
The proof proceeds inductively. 
As the base case, by the discreteness of the flat length spectrum we begin with the shortest saddle connection. 
Here, we are using that $Z$ is non-empty.

For the induction step, suppose we have found $k > 0$ saddle connections with pairwise disjoint interiors. 
We cut along the union of their closures. 
Let $C$ be a component of the resulting open surface. 
Let $\cover{C}$ be the closure in the path metric of the universal cover of $C$.
By the Gauss--Bonnet theorem, the surface $\cover{C}$ has at least three singularities in its boundary.
If $\cover{C}$ has exactly three singularities in its boundary then $C$ is already a triangle and we are done with $C$.
So suppose instead that $\cover{C}$ has at least four singularities on its boundary.

There are now two cases.
Either some singularity on $\partial \cover{C}$ has an interior angle less than $\pi$ or all interior angles are at least $\pi$.
Suppose that the interior angle at a singularity $y$ is less than $\pi$. 
Suppose that $x$ and $z$ are the singularities adjacent to $y$. 
By considering the developing map to the plane, we can make sense of the straight line segment $[x,z]$. 
Let $T \subset \cover{C}$ be the largest embedded triangle with 
\begin{itemize}
\item one vertex at $y$, 
\item the side opposite $y$ parallel to $[x,z]$, and 
\item the remaining sides contained in $[y,x]$ and $[y,z]$. 
\end{itemize} 
Either the side of $T$ opposite $y$ contains a singularity $w$ in its interior or it does not. 
In the first case, the segment $[y,w]$ is a saddle connection in $\cover{C}$; in the second case, the segment $[x,z]$ is. 
In either case, the saddle connection descends to $C$ and we are done by induction. 
 
Now suppose that all interior angles are at least $\pi$. 
By the Gauss--Bonnet theorem, this is only possible when $C$ has non-positive Euler characteristic. 
Thus $\cover{C}$ has at least two boundary components $\alpha$ and $\beta$.
Let $\gamma$ be an embedded arc in $\cover{C}$ that connects a singularity on $\alpha$ to a singularity on $\beta$. 
Pull $\gamma$ tight. 
Thus $\gamma$ has a subarc that is a saddle connection in $\cover{C}$ and again we are done by induction.
\end{proof} 

The proof of the following lemma is similar to the proof of Lemma~\ref{l:veering-quads} and is omitted. 

\begin{lemma}
\label{l:straight-quads}
Suppose that $\tau$ is a straight triangulation of a flat surface $q$.
Suppose $e$ is an edge of $\tau$.  
Then the adjacent triangles $T$ and $T'$ are distinct. \qed
\end{lemma}


We note that the edges of a straight triangulation have widths and heights.
However, if the triangulation is not veering we do not get a sensible point in the veering polytope.
This is because for a monochromatic triangle there is no way to deduce, purely from the combinatorics of $\tau$, the triangle equality of widths.


\subsection{Veering triangulations of half-translation surfaces:}

A topological veering triangulation $\tau$ is a \emph{veering triangulation for a half-translation surface} $(S,q)$ if 
\begin{itemize}
\item $\tau$ is straight and 
\item the slopes of the edges match the colours of the edges: that is, red edges have positive slope and blue edges have negative slope.
\end{itemize}
In particular, there are no horizontal or vertical edges. 
Fixing a labelling $\pi \from \calA \to E(\tau)$, a veering triangulation of $(S,q)$ gives a point of the veering polytope.

Our immediate goal is to show that generic quadratic differentials in any connected component $\calC$ of any stratum $\calQ(\kappa)$ admit veering triangulations.   We begin by promoting straight triangulations to veering ones, generally following an idea of Agol~\cite{Ago}.
He in turn was inspired by work of Hamenst\"adt~\cite{HamEnable}.

Suppose that  $(\tau^{(0)}, \pi^{(0)})$ is a straight labelled triangulation.
Suppose that $e$ is the label of some edge.  
Let $Q = Q(e, \tau^{(0)})$.
Let $a$, $b$, $c$, and $d$ be the labels of the edges of $\bdy Q$ oriented anti-clockwise so that $e$ meets the beginning of $a$ and the end of $b$. 
We define the \emph{flip} on $e$ just as we do in the topological setting, with the added condition that $Q$ must be convex.
The label $e$ is \emph{(geometrically) forward flippable} if $e$ is at least as wide as any of $a$, $b$, $c$, and $d$ in $\tau^{(0)}$.  
We define \emph{(geometrically) backward flippable} similarly with respect to height.  

Note that if all edges of $Q$ are the same colour then the diagonal is simultaneously both forward and backward flippable.  
See the left hand side figure in Figure~\ref{f:mono-flips}.

If the colours of $a$, $b$, $c$, and $d$ alternate, then the topological and geometric definitions agree.
See Figure~\ref{fig:flippable_edges}.
If three edges of $\bdy Q$ are red and one is blue there are several possibilities. 
Two of these are shown in Figure~\ref{f:mono-flips}.
\begin{figure}
\includegraphics{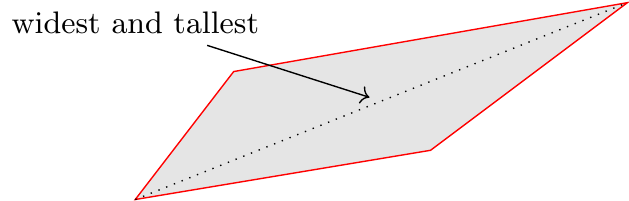}%
\hspace{.1cm}\includegraphics{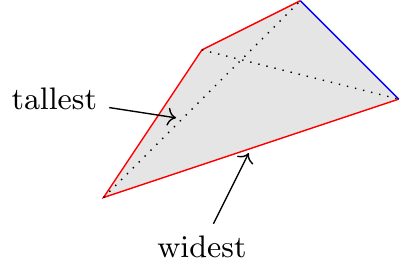}%
\hspace{.1cm}\includegraphics{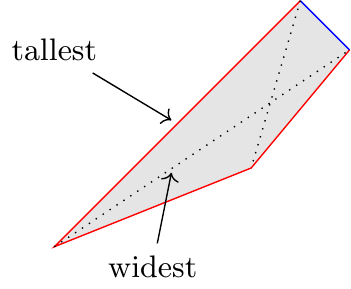}
\caption{The forward and backward flips.}
\label{f:mono-flips}
\end{figure}

Suppose that $e$ flips forward to give a new straight triangulation $\tau^{(1)}$ with the induced labelling $\pi^{(1)}$. 
Breaking symmetry, suppose $\pi^{(1)} (e)$ has positive slope in $\tau^{(1)}$. 
Then 
\begin{equation}\label{e:width-change} 
x^{(1)}_{e} = x^{(0)}_e - x^{(0)}_a - x^{(0)}_c.
\end{equation} 
If $\pi^{(1)}(e)$ has negative slope then we get a similar equation with $a$ and $c$ replaced by $b$ and $d$. 
Also, if we flip backwards from $\tau^{(0)}$ to $\tau^{(-1)}$ then we get similar equations expressing $x^{(0)}$ in terms of $x^{(-1)}$.

We say that a straight triangulation $\tau$ of $(S, q)$ is \emph{staggered} if no edge of $\tau$ is vertical. 

\begin{theorem}
\label{t:staggered}
Suppose that $(S, q)$ is a half-translation surface.  
Suppose that $\{\tau^{(n)}\}_n$ is an infinite forward flip sequence of staggered triangulations.
Then
\begin{enumerate}
\item
the maximal width among edges in $\tau^{(n)}$ tends to zero, 
\item
the minimal height among edges in $\tau^{(n)}$ tends to infinity, 
\item
eventually $\tau^{(n)}$ is veering, and 
\item 
$(S,q)$ is vertically Keane.
\end{enumerate}
Conversely, if $(S,q)$ is vertically Keane then such a staggered flip sequence exists.  
\end{theorem}





\begin{remark}
\label{r:repackage}
Suppose that $(S, q)$ is a half-translation surface.  
Then Theorem~\ref{t:staggered} implies that the following are equivalent. 
\begin{itemize}
\item
$(S, q)$ is vertically Keane.
\item
$(S, q)$ admits a veering forward flip sequence.
\end{itemize}
\end{remark}



\begin{proof}[Proof of Theorem~\ref{t:staggered}]
First we prove that the maximal width among edges in $\tau^{(n)}$ tends to zero as $n$ tends to $\infty$.
From the discreteness of the length spectrum it follows that the shortest height also tends to infinity.

Fix a labelling $\pi^{(0)}$ on $\tau^{(0)}$.
This induces labelings $\pi^{(n)}$ on $\tau^{(n)}$ as discussed near the beginning of Section~\ref{sec:topological-veering-triangulations}.
We define $x^{(n)}_a$ to be the width of the edge $\pi^{(n)}(a)$ in $\tau^{(n)}$.

For each $a$ in $\calA$, the sequence $x^{(n)}_a$ is non-increasing.  
Let $x^{(\infty)}_a$ be its limit.
We define 
\[
\calB = \{ b \in \calA \st x^{(\infty)}_b > 0 \}.
\]
We call the labels of $\calB$ \emph{broad} and the labels of $\calA - \calB$ \emph{narrow}.
If $\calB$ is empty we are done, so suppose otherwise.

We define, for $n > 0$ and for $n = \infty$, 
\[
\Lambda^{(n)} = \sum\limits_{b \in \calB} x^{(n)}_b
\qquad
\text{and}
\qquad
\lambda^{(n)} = \min\limits_{b \in \calB} x^{(n)}_b.
\]
Let $N$ be such that for all $n > N$ we have
\begin{equation}
\label{eq:lambdas}
\Lambda^{(\infty)} < \Lambda^{(n)} < \Lambda^{(\infty)} + \lambda^{(\infty)}.
\end{equation}
Suppose also that for $n > N$ we have $x_a^{(n)} < \lambda^{(\infty)}/2$ for all $a \in \calA - \calB$.

Let $e \in \calB$ be the label of the edge that flips forward from $\tau^{(n)}$ to $\tau^{(n+1)}$.
Suppose that $a$, $b$, $c$, and $d$ are the labels of the edges in $Q(e,\tau)$. Let $T$ and $T'$ be the two triangles that contain $e$.
We may assume that the edges of $T$ and $T'$ are labeled respectively in anti-clockwise order $e$, $a$, $b$ and $e$, $c$, $d$.
Without loss of generality, suppose that $\pi^{(n+1)}(e)$ has positive slope.
By~\ref{e:width-change},
\[
x^{(n+1)}_e = x^{(n)}_e - x^{(n)}_a - x^{(n)}_c.
\]
If either $a$ or $c$ is broad then the sum $\Lambda^{(n)}$ drops by more than $\lambda^{(\infty)}$ which contradicts
equation~\eqref{eq:lambdas}.
Hence, both $a$ and $c$ are narrow. We deduce that $\calA \setminus \calB$ is non-empty.
Since $S$ is connected, there is a triangle $T$ in $\tau^{(n)}$ such that $T$
has at least one broad label, say $b$, and at least one narrow label, say $a$.
The triangle equality~\eqref{eq:triangle-equality} for $T$
cannot be satisfied if its third edge is narrow.
So $T$ is a triangle with two broad edges and one narrow edge.

Suppose by induction that at stage $m$ we have a triangle $T^{(m)}$ with two broad edges and one narrow edge. 
The narrow edge can not flip (by our assumption on the flip sequence). 
If one of the broad edge flips, then in the next triangulation the narrow edge is again in some triangle $T^{(m+1)}$ with two broad edges.

We have proven that the narrow edge $a$ never flips.
Hence $x^{(\infty)}_a = x^{(n)}_a > 0$, a contradiction.

This concludes the proof of the first (and thus the second) conclusion.

We now prove the third conclusion that $\tau^{(n)}$ is eventually veering.  
Note that $(S, q)$ has only finitely many horizontal saddle connections.  
Let $w_h$ be the width of the shortest of these.  
Since the widths tends to zero, there is a $N'$ so for all $n > N'$, we have all widths less than $w_h$.

From now on we suppose that $n > N'$.
In particular $\tau^{(n)}$ has no horizontal edge.
Since $(S,q)$ is vertically Keane we colour each edge $e$ of $\tau^{(m)}$ red or blue as the slope of $e$ is positive or negative.

\begin{claim*}
The number of monochromatic triangles in $\tau^{(n)}$ is a non-increasing sequence. 
\end{claim*}

\begin{proof}
Suppose that $e$ is being flipped in $\tau^{(n)}$.
Let $Q = Q(e,\tau^{(n)})$ and lay $Q$ out in the plane. 
By the hypothesis, $Q$ is contained in the smallest vertical strip
containing $e$.

Let us distinguish cases by the colours we find on $\bdy Q$.
If all edges in $\bdy Q$ are the same colour, then $e$ is again that colour (because it is widest). Then either the number of 
monochromatic triangles stay the same or goes down by two in $\tau^{(n+1)}$. 

If three edges are red and one is blue, then $e$ must be red.
So $Q$ is made of one monochromatic and one non-monochromatic triangle.
After flipping, we can not have two monochromatic triangles because of
the colours on $\bdy Q$.
If three edges are blue and one is red, then a similar argument holds.

Now, let us assume that $\bdy Q$ is made of two blue and two red edges.
The colours must alternate on $\bdy Q$. 
Thus we have no monochromatic triangle before or after the flip.
\end{proof}

We take $n > N'$, as above.  
If $\tau^{(n)}$ has no monochromatic triangles, then since it is straight, it is veering.
So suppose otherwise. 

Fix a monochromatic triangle $T^{(n)}$ in $\tau^{(n)}$ whose width is the smallest. 
Let $e^{(n)}$ be the widest edge of $T^{(n)}$.
Breaking the symmetry, we may assume that $T^{(n)}$ is red. 
Let $a^{(n)}$ be the non-wide edge in $T^{(n)}$, anti-clockwise from $e^{(n)}$. 
To allow us to draw a picture, we say the side of $a^{(n)}$ containing $T^{(n)}$ is the \emph{lower} side.
Let $\tau^{(m)}$ be the first triangulation along the flip sequence with $m > n$ and with $a^{(m)}$ wide to its lower side.
This $m$ exists by the first conclusion of the proposition.
Let $T^{(i)}$ for $i$ between $n$ and $m$ be the triangle below $a^{(n)}$ in the triangulation $\tau^{(i)}$. 

\begin{figure}
\includegraphics{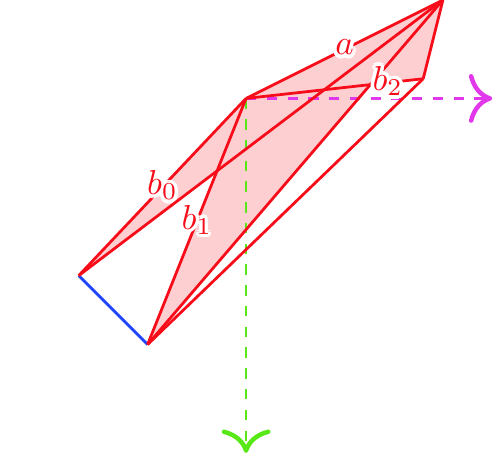}%
\hspace{.1cm}\includegraphics{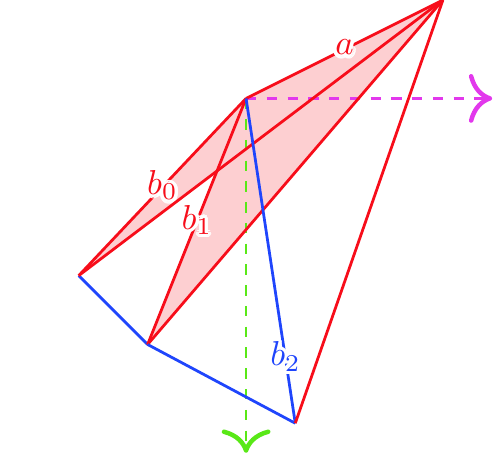}%
\hspace{.1cm}\includegraphics{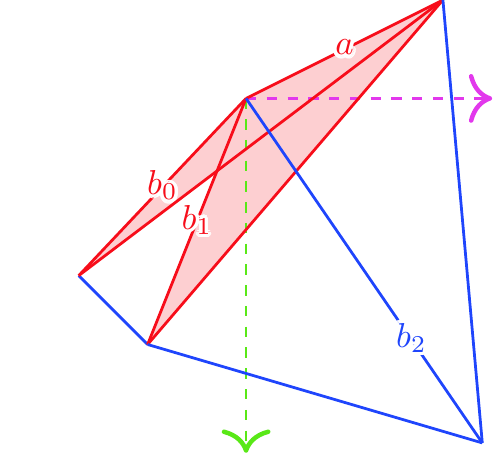}
\caption{The three possible evolutions of a red monochromatic triangle.}
\label{f:evolution}
\end{figure}

There are two possibilities: either the triangles $T^{(i)}$ are all monochromatic, hence red, or one of them is not. 

In the former case, the monochromatic triangle $T^{(m)}$ has smaller diameter than $T^{(n)}$.  
See the left diagram of Figure~\ref{f:evolution}.  
We define $a^{(m)}$ to be the edge of $T^{(m)}$ immediately anti-clockwise from $a^{(n)}$. 
We now continue with $a^{(m)}$ replacing $a^{(n)}$. 
This process cannot repeat forever by the discreteness of the flat length spectrum. 

We now assume that we are in the latter case. 
Let $k$ be the smallest index greater than $n$ so that $T^{(k)}$ is not monochromatic. 
Let $e^{(k-1)}$ be the wide edge in $T^{(k-1)}$. 
Note that $e^{(k-1)}$ flips to a blue edge $e'$ in $\tau^{(k)}$. 
See the middle and right diagrams in Figure~\ref{f:evolution}.
Thus, both triangles in $\tau^{(k)}$ containing $e'$ are non-monochromatic. 
We deduce that the number of monochromatic triangles has decreased and we are done.

We now prove the fourth conclusion that $(S,q)$ is vertically Keane. 
Suppose that $(S,q)$ has a vertical saddle connection. 
Call this saddle connection $\beta$.  
Let $\ell_q(\beta)$ be its height. 
Let $z$ and $z'$ be the endpoints of $\beta$. 
We orient $\beta$ away from $z$.  

Let $T^{(n)}$ be the triangle in $\tau^{(n)}$ with vertex $z$ containing an initial segment of $\beta$. 
By the second conclusion there is some $N''$ such that for $n > N''$ the saddle connection $\beta$ is completely contained in $T^{(n)}$. 
This implies that the triangle $T^{(n)}$ contains the singularity $z'$ in its interior, a contradiction. 

Finally, we prove that if $(S, q)$ is vertically Keane then it admits a staggered forward flip sequence.
There is a straight triangulation $\tau$ of $(S, q)$ by Lemma~\ref{l:straight}; this must be staggered.
Now any forward flip sequence starting at $\tau$ gives the desired conclusion. 
\end{proof}

Veering triangulations and staggered triangulations are respective analogues of dynamically irreducible generalised permutations and irreducible generalised permutations in the theory of linear involutions. 
See~\cite{Boi-Lan} for the precise definitions and details.

We now give a simple version of \cite[Lemma 1.11]{Del-Ulc}. 
See also the remark immediately after \cite[Theorem 2.1]{Min-Tay}. 

\begin{lemma}
\label{l:edge-rectangle}
Suppose $(S, q)$ is a half-translation surface.  
Suppose that $\tau$ is a straight triangulation of $(S, q)$.  
Then $\tau$ is veering if and only if, for every edge $e$ in $\tau$ there is an immersed rectangle $R(e)$ having $e$ as a diagonal.
\end{lemma}

\begin{figure}[!ht]
\centering \includegraphics{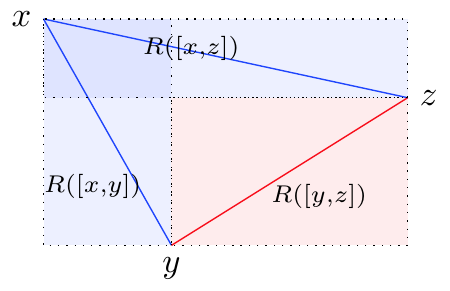}
\caption{In a veering triangulation, each edge is contained in an immersed rectangle (Lemma~\ref{l:edge-rectangle}). Hence each triangle is contained in the immersed rectangle that is the union of the three rectangles around the edges.}
\end{figure}

\begin{proof}
Suppose that $\tau$ is veering and fix an edge $e$.  
Suppose that $z$ and $z'$ are the endpoints of $e$. 
If $R(e)$ does not exist then there is a singularity $z''$ that is close to a point of the interior of $e$. 
We choose $e$ and $z''$ to minimize this distance. 
Note that the vertices $z$, $z'$ and $z''$ cannot be the vertices of a triangle because the triangle would be monochromatic. 
So there must be some edge $e'$ of $\tau$ that locally separates $e$ from $z''$. 
But then $e'$ is even closer to $z''$, a contradiction. 

For the other direction, we note that the given rectangles $R(e)$ prevent monochromatic triangles in $\tau$.  
So we are done. 
\end{proof}

We now translate part of the work of \cite{Del-Ulc, Del-Ulc-2} into our setting.

Suppose $(S,q)$ is Keane and $\beta$ is a vertical separatrix emanating from a singularity $y$. 
Suppose $\tau$ is a veering triangulation of $(S,q)$.  
As a useful piece of notation we take $T(\beta, \tau)$ to be the triangle of $\tau$ containing the germ of $\beta$. 
We define $e(\beta, \tau)$ to be the edge of $T(\beta, \tau)$ opposite the initial singularity of $\beta$.  
We lay out $y$, $\beta$, and $T(\beta, \tau)$ in the plane.
We define $w(\beta, \tau)$ to be the \emph{(vertical) wedge} of $\tau$ about $\beta$: this is the cone in the plane based at $y$ and spanned by the edges of $T(\beta, \tau)$ other than $e(\beta, \tau)$.  
Finally, any triangle $T$ of $\tau$ determines a unique wedge; we denote this by $w(T)$. 

We will need the follow consequence of Lemma~\ref{l:edge-rectangle}.

\begin{lemma}
\label{l:wedges-nest}
Suppose that $(S,q)$ is a Keane half-translation surface.  
Suppose that $y$ is a singularity of $(S,q)$, that $\beta$ is a vertical separatrix at $y$, and $\tau$ are $\tau'$ are veering triangulations of $(S,q)$.  
Then either $w(\beta, \tau) \subset w(\beta, \tau')$ or $w(\beta, \tau') \subset w(\beta, \tau)$.
\end{lemma}

\begin{proof}
Suppose, for a contradiction, that neither of the wedges $w(\beta, \tau)$ and $w(\beta, \tau')$ is contained in the other.  
We lay out $y$ and $\beta$ in the plane, as well as the triangles $T = T(\beta, \tau)$ and $T' = T(\beta, \tau')$.  
Let $x$ and $x'$ be upper left vertices of $T$ and $T'$; let $z$ and $z'$ be the upper right vertices. 
Breaking symmetry, suppose that $z$ is lower than $x'$.  
There are now two cases as $z$ is to the left or right of $z'$.  
\begin{itemize}
\item
Suppose that $z$ is to the left of $z'$.  
Then the rectangle $R([x', z'])$ contains $z$, contradicting Lemma~\ref{l:edge-rectangle}.
\item
Suppose that $z$ is to the right left of $z'$.  
Then the rectangle $R([y, z])$ contains $z'$, contradicting Lemma~\ref{l:edge-rectangle}. \qedhere
\end{itemize}
\end{proof}

\begin{lemma}
\label{l:wedge-order}
Suppose that $\tau$ and $\sigma$ are veering triangulations of $(S,q)$.  
Suppose that $\beta$ is a vertical separatrix of $(S,q)$ so that $w(\beta, \sigma)$ is strictly contained in $w(\beta, \tau)$ and so that $e = e(\beta, \tau)$ is the widest edge among all wedges with that property.  
Then $e$ is forward flippable. 

Furthermore, if $\tau'$ is the veering triangulation obtained by flipping at $e$, and if $\gamma$ is any vertical separatrix, then 
\[
w(\gamma, \sigma) \subset w(\gamma, \tau) \qquad \mbox{if and only if} \qquad w(\gamma, \sigma) \subset w(\gamma, \tau').
\]
\end{lemma}

\begin{proof}
Suppose for a contradiction that $e$ is not forward flippable.  
Set $T = T(\beta, \tau)$ and let $T'$ be the triangle of $\tau$ so that $Q(e, \tau) = T \cup T'$.  
Thus the width of $T'$ is greater than that of $e$.  
Let $\beta'$ be the vertical separatrix for $T'$.  
By our assumption on $\beta$, and by Lemma~\ref{l:wedges-nest}, the wedge $w' = w(\beta', \sigma)$ contains $T'$.  
Thus there is an edge $d$ of $\sigma$ that emanates from the singularity of $w'$ and crosses $\beta \cap T$.  
Thus $d$ crosses an edge of $T(\beta, \sigma)$, a contradiction. 

We now consider $\tau'$, the veering triangulation obtained by flipping at $e$; let $e'$ be the resulting backwards flippable edge. 
Let $T''$ and $T'''$ be the triangles of $\tau'$ so that $Q(e', \tau') = T'' \cup T'''$.  
We arrange matters so that $T''$ contains an initial segment of $\beta$.  
Let $\beta'$ be the vertical separatrix for $T'$ and thus for $T'''$.  
We deduce from Lemma~\ref{l:wedges-nest} that 
\[
w(\beta, \sigma) \subset w(\beta, \tau').
\]
Note that if $w(\beta', \tau') \subset w(\beta', \sigma)$ then, applying Lemma~\ref{l:edge-rectangle}, we find that an edge of $\sigma$ must cross an edge of $T(\beta, \sigma)$; this is a contradiction.  
So we apply Lemma~\ref{l:wedges-nest} again to find
\[
w(\beta', \sigma) \subset w(\beta', \tau')
\]
as desired.
\end{proof}

Suppose that $\tau$ and $\tau'$ are veering triangulations of $(S,q)$.  
Suppose further that for every vertical separatrix $\beta$ in $(S,q)$ we have $w(\beta, \tau') \subset w(\beta, \tau)$.  
Then we write $\tau <_q \tau'$ and say that $\tau$ \emph{precedes} $\tau'$.  
For example, if $\tau'$ is a forward flip of $\tau$ then we have $\tau <_q \tau'$. 
Let $\calV_q$ be the set of veering triangulations of $(S,q)$; so $<_q$ gives $\calV_q$ the structure of a poset.  

\begin{corollary}\label{c:forward}
If $\tau <_q \sigma$ then there is a forward flip sequence from $\tau$ to $\sigma$. \qed
\end{corollary}

We now give a structural result similar in spirit to \cite[Lemma 5.1]{HamEnable}.
Our result is strictly stronger than the connectedness result of~\cite[Proposition~3.3]{Min-Tay};
also, we do not use Delauney triangulations.  

Recall that a poset with unique joins and meets is called a \emph{lattice}.
If $\sigma$ and $\tau$ are elements of the lattice, then $\sigma \join \tau$ is the standard notation for their join while $\sigma \meet \tau$ is the notation for their meet.

\begin{proposition}
\label{p:v-q-lattice}
Suppose that $(S,q)$ is a half-translation surface.  
Then $(\calV_q, <_q)$ is naturally a sublattice of $\ZZ^{\calA}$.
\end{proposition}

\begin{proof}
We first prove that $(\calV_q, <_q)$ is a lattice.
Suppose that $\tau$ and $\sigma$ are veering triangulations of $(S,q)$. 
Set $\tau^{(0)} = \tau$. 
Suppose that the wedge $w(\beta, \tau^{(0)})$ strictly contains $w(\beta, \sigma)$ and among all such wedges with this property the triangle $T(\beta, \tau^{(0)})$ is the widest.
Thus the edge $e(\beta, \tau^{(0)})$ is forward flippable. 
Flipping $e(\beta, \tau^{(0)})$ forward gives a triangulation $\tau^{(1)}$ that covers $\tau^{(0)}$.
We repeat this process replacing $\tau^{(0)}$ by $\tau^{(1)}$. 
By the discreteness of the flat length spectrum we arrive at a triangulation $\tau^{(n)}$ so that no wedge of $\tau^{(n)}$ strictly contains a wedge of $\sigma$. 
See Figure~\ref{f:wedges}.

\begin{figure}
\includegraphics{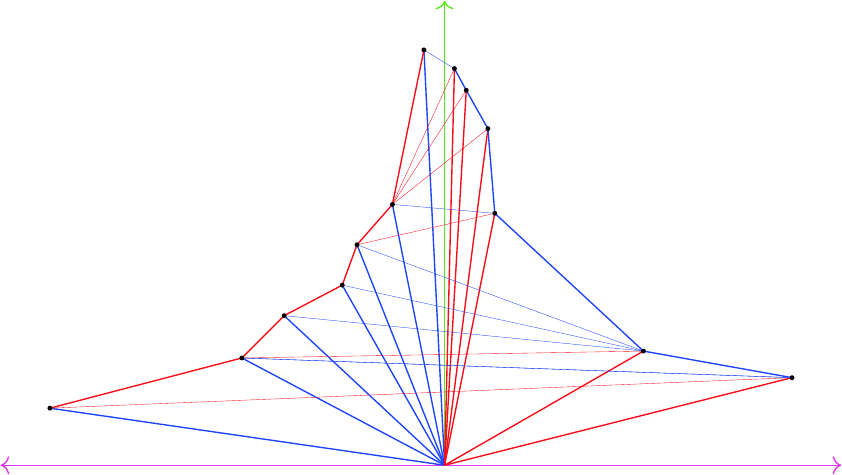}%
\caption{Wedges nesting around their vertical separatrix.}
\label{f:wedges}
\end{figure}

Similarly, we can flip forward from $\sigma^{(0)} = \sigma$ to a triangulation $\sigma^{(m)}$ so that no wedge of $\sigma^{(m)}$ strictly contains a wedge of $\tau$. 
Since $\tau^{(n)}$ and $\sigma^{(m)}$ have the same wedges, they are equal.
This implies that joins exist and are unique.

Meets are constructed similarly replacing forward flips by backward flips and reversing the direction of the strict containment.

We next prove that $(\calV_q, <_q)$ is distributive.
Let $\rho$, $\sigma$, and $\tau$ be veering triangulations of $(S,q)$. 
Let $\beta$ be any vertical separatrix. 
Let $w_\rho = w(\beta, \rho)$ be the wedge of $\rho$ about $\beta$ and define $w_\sigma$ and $w_\tau$ similarly.
Note that 
\[
(w_\rho \cap w_\sigma) \cup w_\tau = (w_\rho \cup w_\sigma) \cap (w_\sigma \cup w_\tau)
\]
by the distributivity of union over intersection.  
Thus $(\rho \join \sigma) \meet \tau$ and $(\rho \meet \tau) \join (\sigma \meet \tau)$ have the same wedge about $\beta$.  
Since a triangulation is determined by its wedges, we are done.

We now prove that $(\calV_q, <_q)$ is a sublattice of $\ZZ^{\calA}$.
As a bit of notation, for $a \in \calA$ we define $a^\ast \from \calA \to \ZZ$ taking $a$ to $1$ and all other labels to $0$.
We now inductively build an inclusion $\Phi \from \calV_q \to \ZZ^{\calA}$.
Fix a basepoint $\tau_0 \in \calV_q$ and fix a labelling of it. 
Because $\calV_q$ is a distributive lattice, the labelling of $\tau_0$ induces a well-defined labelling on every vertex of $\calV_q$.

As the base case, we define $\Phi(\tau_0) = 0$.
Suppose $\tau$ is any vertex in $\calV_q$ and $\Phi(\tau) = u$. 
Suppose $\tau'$ is connected to $\tau$ by a single flip on the edge labelled $a$. 
Then we define $\Phi(\tau') = u \pm a^\ast$ where the sign is positive or negative as the flip is forward or backward.
Because $\calV_q$ is a distributive lattice, the function $\Phi$ is well-defined.
\end{proof}

\subsection{Balanced triangulations:}
\label{s:balanced}

By Remark~\ref{r:repackage}, a half-translation surface $(S,q)$ can have infinitely many veering triangulations.
In the opposite direction, to go from a topological veering triangulation (equipped with widths and heights) to a half-translation surface we must restrict the allowed triangulations and parameters.
We proceed as follows.

Let $\tau$ be a labelled topological veering triangulation. 
As defined in Section~\ref{s:veering-polytope}, let $\hor_\tau$ and $\ver_\tau$ be the width and height polytopes, respectively.  
Also, recall that $C_\tau = \hor_\tau \cross \ver_\tau$ is the veering polytope.
We say that a topological veering triangulation $\tau$ is \emph{core} if 
\[
\dim_{\RR} \hor_\tau = \dim_{\RR} \ver_\tau = \dim_{\CC}(\calQ(\kappa)),
\]
that is, the future ($\hor_\tau$) and the past ($\ver_\tau$) have the expected dimension. 

If $\tau$ is core then any $(x, y) \in C_\tau$ gives a half-translation surface which in turn defines a quadratic differential $q_\tau (x,y)$ lying in some stratum component $\calC \subset \calQ(\kappa)$.  
As usual $\kappa$ is the numerical datum determined by $\tau$. 
This gives a map $q_\tau \from C_\tau \to \calC$.
Conversely, if $\tau$ is a straight veering triangulation of some $(S,q)$ then $\tau$ is core.  
To see this, recall that periods of a quadratic differential give local coordinates in $\calC$.  
Thus if two small deformations of $q$ give the same widths and heights, they are in fact the same.  
Thus $q_\tau$ is locally surjective. 

Each core triangulation $\tau = \tau^{(0)}$ is said to be \emph{horizontally balanced} or \emph{balanced} in short, for parameters $(x^{(0)}, y^{(0)})$ if they satisfy:
\begin{enumerate}
\item\label{con1} 
for each forward flippable edge $\pi(a) = e$ we have $x^{(0)}_a \leqslant 1$ and
\item\label{con2} 
for each backward flippable edge $\pi(a) = e$ we have $x^{(-1)}_a > 1$.
\end{enumerate}
Observe that condition (\ref{con1}) implies that all widths in $\tau$ are at most one.

\begin{proposition}\label{p:existence-uniqueness}
Suppose that $(S,q)$ is Keane. Then $(S,q)$ admits a unique balanced triangulation. 
\end{proposition} 

\begin{proof}
We first prove that if $(S,q)$ is Keane then it admits a balanced triangulation.
By Theorem~\ref{t:staggered}, $(S,q)$ admits a veering triangulation $\tau$. 
If $\tau$ is balanced we are done.
So suppose $\tau$ is not balanced.
For convenience, we label $\tau$ and so induce a labelling on any triangulation arising in a flip sequence from $\tau$. 
We will now flip backwards then forwards to arrive at a balanced triangulation. 

The first step is to flip $\tau^{(0)} = \tau$ backwards.
Let $\calB^{(0)}$ be the labels of backward flippable edges in $\tau^{(0)}$ that do not satisfy condition (\ref{con2}).
If $\calB^{(0)}$ is empty we are done with this step.
So suppose not.
We flip back all edges with labels in $\calB^{(0)}$ to arrive at a triangulation $\tau^{(-1)}$.
We iterate this process of flipping backwards all labels in $\calB^{(-k)}$.
We now apply Theorem~\ref{t:staggered} in backwards time. 
This implies that after finitely many iterations all backward flippable labels in $\tau^{(-n)}$ satisfy condition (\ref{con2}).
Now we pass to the second step.

For the second step, we may assume that the backward flippable labels in $\tau^{(0)} = \tau$ satisfy condition (\ref{con2}). 
Let $\calF^{(0)}$ be the labels of forward flippable edges in $\tau^{(0)}$ that do not satisfy condition (\ref{con1}).
If $\calF^{(0)}$ is empty then $\tau^{(0)}$ is balanced and we are done.
So suppose not.
We now proceed as in step one to arrive at a triangulation $\tau^{(m)}$ for which all forward flippable labels satisfy condition (\ref{con1}).
Since $\calV_q$ is distributive, from Proposition~\ref{p:v-q-lattice}, the triangulation $\tau^{(m)}$ is balanced and we are done.

We now show uniqueness. 
Suppose $(S,q)$ admits distinct balanced triangulations $\sigma$ and $\tau$.  
Since $\calV_q$ is a lattice we may define $\rho = \sigma \meet \tau$.
Let $\alpha$ and $\beta$ be forward flip sequences from $\rho$ to $\sigma$ and $\tau$ respectively. 
Since $\sigma \neq \tau$ at least one of $\alpha$ or $\beta$ has positive length. 
Breaking symmetry, we assume that $\alpha$ has positive length.
For convenience, we fix a labelling of $\rho$ to induce a labelling on all triangulations in $\alpha$ and $\beta$.
Let $e$ be the label that flips forward first in $\alpha$. 
If $e$ flips forward anywhere in $\beta$, then by distributivity we may assume that $e$ flips forward first in $\beta$. 
Thus the forward flip of $\rho$ along $e$ is again a lower bound for $\sigma$ and $\tau$, a contradiction.
So we may assume $e$ does not flip forward in $\beta$. 
We deduce that $e$ is forward flippable in $\tau$.
Hence it has width at most one in $\tau$, hence in $\rho$.
On the other hand, $e$ flips forward in $\alpha$. 
Since $\sigma$ is balanced $e$ must have width greater than one.   
We thus arrive at a contradiction.
 \end{proof}

Let $\tau$ be a labelled core triangulation of $(S,Z)$.
Let $T$ be a triangle of $\tau$.
Let $e$ be the label of the widest edge in $T$, let $b$ be the label of the tallest edge in $T$, and let $a$ be the label of the remaining edge.  
Note that $x_e = x_a + x_b$ and $y_b = y_e+ y_a$.
A simple calculation shows that 
\[
2 \Area(T) = x_e y_a + x_a y_b - x_a y_a.
\]
Thus, the area of $(S, q_\tau(x,y))$ can be expressed as a quadratic form $\Omega_\tau$ pairing $x$ and $y$.

The level sets of area in the stratum component $\calC$ are invariant under \teichmuller{} flow. 
From now on we will restrict to unit-area half-translation surfaces.
We denote them by $\calC^1$. 

Let $B_\tau$ be the set of parameters $(x,y)$ such that $\tau$ is balanced for $(x,y)$ and $\Omega_\tau(x,y) =1$. 
Set $\calC^1_\tau = q_\tau(B_\tau) \subset \calC^1$.
This is the set of half-translation surfaces that arise from $B_\tau$.
We say that the triangulation $\tau$ is a \emph{balanced} triangulation if the interior of  $\calC^1_\tau$ contains an open set in $\calC^1$.
This is an actual restriction: not all core triangulations are balanced. 
Figure~\ref{fig:geom_non_bal} shows an example in the stratum $\mathcal{H}(2)$.

\begin{figure}[ht]
\centering
\includegraphics{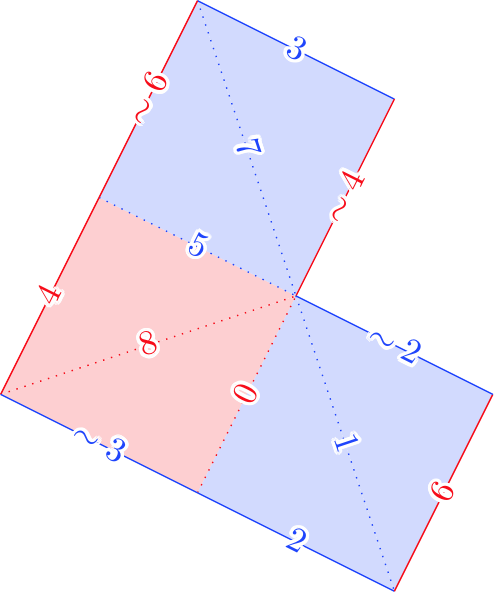}
\label{fig:geom_non_bal}
\caption{A core triangulation that is not balanced. 
There are nine edges each labelled by an integer in $\{0, 1, \cdots, 8\}$. 
All pairings are by translations and indicated by a tilde.}
\end{figure}

As a direct corollary of Proposition~\ref{p:existence-uniqueness}, we get
\begin{corollary}\label{c:full-measure}
\mbox{}
\begin{enumerate}
\item Up to excision of a set of measure zero from the left hand side
\[
\calC^1 = \bigcup\limits_{\tau \text{ balanced}} \calC^1_\tau 
\] 
\item For distinct balanced triangulations $\tau$ and $\sigma$ we have $\calC^1_\tau \cap \calC^1_\sigma = \emptyset$. \qed
\end{enumerate}
\end{corollary} 


\section{teichmuller{} flow}
\label{s:model}

To construct a measurable model for the \teichmuller{} flow on $\calC^1$, we glue the pieces $B_\tau$ along their boundaries by flip data. 

The model is easily described in words.
Let $\sigma$ be a balanced triangulation and $(x, y) \in B_\sigma$.
Let $\psi_t(x,y) = (e^t x, e^{-t} y)$.  
Let $\xi$ be the smallest (non-negative) time at which some forward flippable edge in $\sigma$ has width one.
At this time, we may flip all forward flippable edges with width one to get a new triangulation $\tau$.
For a subset of $B_\sigma$ with full measure, the new triangulation $\tau$ is balanced. 
For example, this happens if the half-translation surface $(S,q_\sigma(x,y))$ is vertically Keane. 
Note that the resulting widths and heights for $\tau$ are in $B_\tau$.
So we may continue the flow $\psi_t$ in $B_\tau$.
We now describe the model in detail.

\subsection{The space} 
Let $\tau$ be a balanced triangulation. 
We define the \emph{upper boundary} $\up_\tau$ of $B_\tau$ to be those $(x,y) \in B_\tau$ for which the width of some forward flippable edge is one. 
The \emph{lower boundary} $\low_\tau$ is the set of those $(x,y) \in C^1_\tau$ for which the width of some backward flippable edge becomes one after it is flipped back.
Let $\overline{B_\tau}$ be the closure of $B_\tau$ in $C^1_\tau$.
Note that  $\overline{B_\tau} = B_\tau \cup \low_\tau$.

Let $E_\tau$ be the set of $(x,y) \in \overline{B_\tau}$ such that $(S,q_\tau(x,y))$ is not Keane. 
The set $E_\tau$ has measure zero and we will excise it.
Let $\calW$ be the disjoint union of the sets $\overline{B_\tau} \setminus E_\tau$.
We define an equivalence relation $\sim$ on $\calW$ as follows. 
Let $\tau$ and $\tau'$ be distinct balanced triangulations.
We say that a point $(x,y) \in \up_\tau$ is equivalent to a point $(x',y') \in \low_{\tau'}$ if and only if $q_\tau(x,y) = q_{\tau'}(x',y')$. 
Note that flipping all forward flippable edges with width one in $(S,q_\tau(x,y))$ gives $\tau'$ and the new parameters are $(x',y')$. 
Let $\calD = \calW/ \sim$. 

By Proposition~\ref{p:existence-uniqueness}, the maps $q_\tau$ glue together to give an injective map $q \from \calD \to \calC^1$.
By Corollary~\ref{c:full-measure}, the image $q(\calD)$ has full measure in $\calC^1$.

\subsection{The flow}
We now define a flow $\psi_t$ on $\calD$ that is conjugate to the \teichmuller{} flow $\phi_t$ on $q(\calD)$.  
Let $\sigma$ be a balanced triangulation.
Let $(x, y) \in \overline{B_\sigma}$.
Consider the interval of times $t \geq 0$ such that $e^t x_r \leq 1$ for all $r \in \calA$. 
Let $\xi \geq 0$ be the right endpoint of this interval.
Note that $\xi$ depends only on $x$.
So we may denote it as $\xi(x)$. 
Note that if $(x, y)$ is in $\low_\sigma$ then $\xi(x) > 0$. 

We define
\[
\psi_t (x, y) \defeq (e^t x, e^{-t} y).
\]
Note that for all $0 \leq t < \xi$ we have $\psi_t(x,y) \in \overline{B_\sigma}$ and $\psi_\xi(x,y) \in \up_\sigma$.
By construction, $\psi_\xi (x,y)$ is equivalent to a point in $\low_\tau$ where $\tau$ is obtained by flipping all forward flippable edges with width one in $(S,q(\psi_\xi (x,y))$.
By replacing $\sigma$ by $\tau$ we may continue the flow.
We can also flow backwards in time and reverse flips.
We thus get a well-defined flow $\calD \times \RR \to \calD$.
We denote this flow by $\psi_t: \calD \to \calD$.
Compatibility of the identifications implies $q(\psi_t (x,y) ) = \phi_t(q(x,y))$, that is the flow $\psi_t$ conjugates to the \teichmuller{} flow $\phi_t$.

\subsection{The transversal:}
Consider the image $\calT$ in $\calD$ of the union of upper boundaries. 
The number $\xi$ defined above gives the first return time to $\calT$.
We regard the first return time as a function $\xi \from \calT \to \RR_{> 0}$. 

This expresses the flow $\psi$ as a suspension flow over the first return map on $\calT$ with the return time $\xi$ as the roof function.
Moreover, the dynamics on $\calT$ is itself a skew product over the dynamics on the width parameters.
In this respect, the theory is structurally similar to the theory of interval exchange maps.
See Avila--Gouëzel--Yoccoz~\cite{Avi-Gou-Yoc}.

\subsection{The core graph}
We give a direct application of the model. The \emph{core graph} $G(\calC)$ is the directed graph with core triangulations as vertices and (geometric) forward flips as edges. Note that a geometric forward flip of a core triangulation gives a core triangulation. We show: 

\begin{lemma}\label{l:strong-connect}
The core graph $G(\calC)$ is strongly connected: given any pair of core triangulations $\sigma$ and $\tau$ there is a directed path from $\sigma$ to $\tau$.
\end{lemma}

\begin{proof}
Suppose first that $\sigma$ and $\tau$ are balanced.
The \teichmuller{} flow $\phi_t$ and hence the flow $\psi_t$ is ergodic \cite{Mas, Vee, Vee2}. 
The sets $B_\sigma$ and $B_\tau$ have positive measure.
By the Birkhoff ergodic theorem, we can find $(x,y)$ in $B_\sigma$ and a time $T > 0$ such that $\psi_T(x,y)$ is in $B_\tau$. 
For any intermediate time $t$, let $\tau(t)$ be the balanced triangulation such that $\psi_t(x,y)$ is contained in $B_{\tau(t)}$. 
Note that $\tau(0) = \sigma$ and $\tau(T) = \tau$.
By compactness of the interval $[0,T]$ finitely many balanced triangulations arise as $\tau(t)$.
These triangulations give a forward flip sequence of balanced triangulations from $\sigma$ to $\tau$.

Now suppose that $\sigma$ is a core triangulation.
Let $(x,y) \in \calC_\sigma$ be any Keane parameters.
Let $u > 0$ be large enough so that the parameters $(e^u x, e^{-u} y)$ satisfy condition (\ref{con2}) of being balanced. 
We now flip forward till condition (\ref{con1}) of being balanced is satisfied and thus arrive at a balanced triangulation.

By reversing the direction we similarly construct a backward flip sequence to a balanced triangulation and we are done.
\end{proof}

Let $q \in \calC$. We may label some $\sigma \in \calV_q$ and so induce a labelling on any $\tau \in \calV_q$. The proof of strong connectivity shows that for generic $q$, the core graph $G(\calC)$ is the quotient of $\calV_q$ that forgets the labelling. 

\subsection{A sample calculation}

We consider the moduli space of unit area flat tori, that is, the unit tangent bundle $T^1 \mathcal{M}_1$ of the modular surface $\mathcal{M}_1 = \mathbb{H} / \SL(2,\ZZ)$.
A veering triangulation on a flat torus has a single wide edge.
Fixing labels, let $c$ be the label of the wide edge with labels $a$ and $b$ anti-clockwise from $c$. 
Then $a$ is blue and $b$ is red and $x_a + x_b = x_c$.
The transversal we have described above corresponds to the constraint $x_c = 1$ which is equivalent to $x_a + x_b = 1$.
Up to relabelling, there are exactly two veering triangulations on flat tori: $c$ is red or $c$ is blue.
When $c$ is red $y_b > y_a $, and when $c$ is blue $y_b < y_a$.

We consider the case when $c$ is red and flips forward to a red edge.
This happens when $x_b > x_a$.
Then $b$ becomes wide and we have returned to the same combinatorics up to relabelling.
If $c$ flips to a blue edge then $a$ becomes wide and we have changed combinatorics.
Similar two cases hold when $c$ is blue.
To carry out computations we focus on the possibility when $c$ is red and $b$ becomes wide.

Gathering all the constraints we get
\begin{enumerate}
\item Horizontal constraints: $x_a + x_b = x_c = 1$ and $x_b > x_a$.
\item Vertical constraints: $y_b = y_a + y_c$ and $y_b > y_a$.
\item Unit Area: $x_a y_b + x_b y_a = 1$.
\end{enumerate}
The volume of $T^1 \mathcal{M}_1$ can be expressed as the iterated integral
\[
\vol\left( T^1 \mathcal{M}_1 \right) = \iint \xi(x_a, x_b, x_c) d\nu_y d\nu_x,
\]
where $d \nu_x$ and $d \nu_y$ are the conditional measures on the horizontal and vertical spaces and
\[
\xi(x_a, x_b, x_c) \defeq 2 \log \left( \frac{1}{x_b}\right) = - 2 \log x_b.
\]
Note that for fixed horizontal parameters we get the (open) interval in vertical parameters between the points
\[
\left(1,1, 0 \right) \text{ and } \left( 0, \frac{1}{x_a}, \frac{1}{x_a} \right).
\]
Integrating the $y$-coordinates we thus get the contribution $\sqrt{3}/2x_a$.
Finally, if we express everything in terms of $x_a$ then we get the integral
\[
I_1 = \int\limits_0^{1/2} - \frac{2 \sqrt{3} \log (1- x_a)}{2 x_a} (\sqrt{3} \, d x_a) = \int\limits_0^{1/2} - \frac{3 \log (1- x_a)}{x_a} d x_a.
\]
In the case when $c$ is blue and $b$ is wide after $c$ flips, the thing that changes is that for fixed horizontal parameters we get the (open) interval in vertical parameters between the points
\[
\left(1, 1, 0 \right) \text{ and } \left( \frac{1}{x_b}, 0, \frac{1}{x_b} \right).
\]
Integrating the $y$-coordinates we thus get the contribution $\sqrt{3}/2x_b$.
If we express everything again in terms of $x_a$ then we get the integral
\[
I_2 = \int\limits_0^{1/2} - \frac{2 \sqrt{3} \log (1- x_a)}{2 x_b} (\sqrt{3} \, d x_a) = \int\limits_0^{1/2} - \frac{3 \log (1 - x_a)}{(1 - x_a)} d x_a.
\]
Now note
\[
I_1 + I_2 = \int\limits_0^{1/2} - \frac{3 \log (1- x_a)}{x_a(1-x_a)} d x_a \asymp \frac{\pi^2}{12},
\]
where the second equality is up to a scale factor (in this case 3).
By symmetry
\[
\vol (T^1 \mathcal{M}_1) = 2 (I_1+ I_2) = \frac{\pi^2}{6},
\]
as expected.

\section{Desired dynamics of the coding}
\label{s:desires}

In this section, we state the precise dynamical properties we want our coding to satisfy. 
We will keep the discussion general but briefly and informally explain how these concepts emerge in our coding. 
The exact details are presented at the end of Section~\ref{s:estimates}.

Let $\Pi$ be a finite or countable set. 
Let $\Sigma = \Pi^\ZZ$ be a symbolic space over the set $\Pi$ endowed with the left shift map $\mathrm{S}$. 
Suppose $w \in \Pi^m$ is a finite word.
The (forward) cylinder $\Sigma(u)$ induced by $u$ is defined as $\Sigma(u) = \{ a \in \Sigma \text{ such that } a_k = u_k \text{ for } k = 0, \dotsc, m - 1 \}$. Given another word $v \in \Pi^n$, we write $uv \in \Pi^{m+n}$ for the concatenation of $u$ and $v$.

\begin{definition}\label{d:bdp}
	We say that an $\mathrm{S}$-invariant probability measure $\mu$ has \emph{bounded distortion} if there exists a constant $K > 0$ such that, for any finite words $u \in \Pi^m$ and $v \in \Pi^n$,
	\[
		\frac{1}{K} \mu(\Sigma(u))\mu(\Sigma(v)) \leqslant \mu(\Sigma(u v)) \leqslant K \mu(\Sigma(u))\mu(\Sigma(v)).
	\]
\end{definition}
The bounded distortion property allows us to treat the symbolic space ``almost'' as a Bernoulli shift. 
For this reason, it is also called an \emph{approximate product structure}. Since $\mu$ is shift-invariant, the previous definition would not change if we used backward or centred cylinders instead of forward cylinders.

In the case of the coding of the Teichmüller flow by interval exchange maps, the bounded distortion property is well-known to hold. This fact is an essential ingredient of the proof of simplicity of the Lyapunov spectra of almost every translation flow \cite{Avi-Via}.

Associated with each symbol $\gamma \in \Pi$ is a parameter space $U_\gamma$. 
Up to excising sets of measure zero, the parameters for distinct symbols are disjoint.
The union $\bigcup_{\eta \in \Pi} U_\gamma$ is a full-measure subset of some polytope $U$.

For each $\gamma \in \Pi$, there is an expanding diffeomorphism $U_\gamma \to U$. One can collect such diffeomorphisms to obtain an expanding map $A \colon \bigcup_{\gamma \in \Pi} U_\gamma \to U$. There exists a unique $A$-invariant absolutely continuous probability measure $\nu$, which is automatically ergodic and even mixing. See \cite{Aar} and \cite[Section 2]{Avi-Gou-Yoc}. The partition $U = \bigcup_{\gamma \in \Pi} U_\gamma$ induces an equivariant bijection between $(\Pi^\ZZ, \mathrm{S})$ and $(U, A)$. The measure $\mu$ that we will consider on $\Sigma^\ZZ$ is the unique $\mathrm{S}$-invariant probability measure rendering this bijection a measure-theoretical conjugation.

A \emph{roof function} is a function $\xi \colon U \to \RR_{\geq 0}$. 
A suspension of the symbolic space is a space that is homeomorphic to $\Sigma \times [0,1]$ with the identification $(a, 0) \sim (\mathrm{S}(a), 1)$.
The roof function equips the suspension with a flow $\Psi$ which flows in the interval direction and satisfies $\psi_{\xi(a)} (a, 0) = (\mathrm{S}(a), 1)$. 
Note that the copy $\Sigma \times \{0 \}$ of the symbolic space is a transversal to the flow.

\begin{definition}
	A roof function is said to have \emph{exponential tails} if there exists $h > 0$ such that
	\[
		\int_\Delta e^{h\xi} d\mu < \infty.
	\]
\end{definition}
The property of having exponential tails implies, in particular, that the volume of the suspension with respect to the local product measure $d\mu \times dt$ is finite.

For \teichmuller{} flow, widths and heights are local co-ordinates for the unstable and stable manifolds. 
So it makes sense to consider transversals that have a local product structure in the widths and heights, and admit first return maps preserving this structure.
This eases the analysis of distortion for the first return by studying the distortion of factor maps.

The transversal $\calT$ constructed above has this local product structure.
The factor maps are given by the same flip matrix.
So it suffices to analyse distortion of just the factor map on widths.
However, $\calT$ is too big and even for the torus its factor map fails the bounded distortion property.
So we have to accelerate the renormalisation through the flip sequences with bad distortion.
In this restricted sense, the setup has the same problems as the theory of interval exchange maps.

In the interval exchange theory, a smaller pre-compact transversal is constructed by a Rauzy sequence with a positive matrix.
The construction of the transversal has to be followed with estimates that eventually prove exponential tails for the roof function.
See Avila--Gouëzel--Yoccoz~\cite{Avi-Gou-Yoc} for the optimal analysis.
The optimal analysis is significantly harder in the theory of linear involutions.
See~\cite{Boi-Lan},~\cite{Avi-Res}.

With veering triangulations, the combinatorial complexity is larger.
So we do something similar but in an indirect way.
Also we do not attempt optimal estimates.
Instead, we prove simpler estimates. 
The interval exchange versions of these estimates are due to Kerckhoff~\cite{Ker}. 
See also \cite{Gad}, \cite{Avi-Res} for the linear involution versions. 
These estimates are sufficient for the purpose.
See the remark at the beginning of~\cite[Appendix A]{Avi-Gou-Yoc}. 

\section{Verifying the dynamics}
\label{s:estimates}

\subsection{Flip matrices}

We first set up some matrix formalism that simplifies subsequent discussion.
Let $\tau$ be a balanced triangulation.
We label $\tau^{(0)} = \tau$ and so induce a labelling on any triangulation arising in a flip sequence from $\tau$.
We denote by $x$ and $y$ the \emph{column} vectors in $\RR^{\calA}_{\geq 0}$ defined by the widths and heights respectively.
Let $E_{jk}$ be an elementary matrix with indices in $\calA$ whose entries are all zero except the $(j,k)^{\text{th}}$ entry which is one.
Let $I$ be the identity matrix.

Let $e$ be the label of a forward flippable edge.
Let $a$, $b$, $c$, and $d$ be the label of edges in $Q(e)$ anti-clockwise from $e$.
If $x^{(0)}_d > x^{(0)}_a$ then $e$ flips to a red edge. 
The new column of widths $x^{(1)}_e$ is given by~\ref{e:width-change}. 
There is a similar relation that gives the new heights $y^{(1)}_e$.
These relations are succinctly expressed as $x^{(0)} = A \, x^{(1)}$ and $y^{(1)} = A \, y^{(0)}$, where $A = I + E_{e a} + E_{e c}$.

On the other hand, if $x^{(0)}_a > x^{(0)}_d$ then $e$ flips to a blue edge.
The new columns of widths and heights $x^{(1)}_e, y^{(1)}_e$ are again expressed as $x^{(0)} = A \, x^{(1)}$ and $y^{(1)} = A \, y^{(0)}$ where $A = I + E_{e b} + E_{e d}$.

The area forms are compatible, that is $A^T \Omega_\sigma = \Omega_\tau A$ where $A^T$ denotes the transpose of $A$.
Recursively, given a flip sequence $\gamma \from \tau^{(0)} \to \tau^{(1)} \to \ldots \to \tau^{(n)}$ we get a non-negative matrix $A_\gamma = A_1 A_2 \dots A_n$ where $A_j$ is the matrix defined as above for the $j$-th flip.
The new parameters satisfy $x^{(0)}  = A_\gamma \, x^{(n)}$ and $y^{(n)} = A_\gamma \, y^{(0)}$.
Also recursively, the area forms are compatible, that is, they satisfy $A^T_\gamma \Omega_{\tau^{(n)}} = \Omega_{\tau^{(0)}} A_\gamma$.

\subsection{Return maps} 

The first return map to $\calT$ can be described piecewise as follows. 

Let $\sigma$ be a balanced triangulation and $e$ be a forward flippable edge. 
Let $\up_\sigma(e)$ be the subset of the upper boundary for which the width of $e$ is one. 
After excising measure zero sets if required, either $\up_\sigma(e) \cap \up_\sigma(f) = \emptyset$ or $\up_\sigma(e) = \up_\sigma(f)$ for distinct forward flippable edges $e$ and $f$.
The latter possibility can occur if a constraint in relative homology forces $e$ and $f$ to be always parallel and equal in length. 
We may further partition $\up_\sigma(e)$ as $\up_\sigma(e, r)$ and $\up_\tau(e, b)$ depending on whether $e$ flips to a red or a blue edge respectively.

Suppose $e$ flips to a red edge.
Let $\tau$ be the resulting balanced triangulation. 
We denote the flip by $\gamma$.
We can then flow by $\psi$ in $B_\tau$ till we arrive at a point in the upper boundary $\up_\tau$.
This gives us a piece of the first return to $\calT$. 
Using flip matrices, the parameters $(x,y)$ in $\up_\sigma(e,r)$ can be expressed in terms of parameters $(x', y')$ in $\up_\tau$ as
\[
x = \frac{1}{e^{\xi(x)}} A_\gamma x'  \text{ and } y = e^{\xi(x)} A^{-1}_\gamma y'.
\]
We think of this piece of the return map as a map $R_\gamma$ from a subset of $\up_\tau$ to $\up_\sigma (e,r)$.  
There are finitely many such maps $R_\gamma$ that arise. 
Let $P_\gamma$ be the map $R_\gamma$ restricted to the widths.

Inductively, the $n^{\text{th}}$ return to $\calT$ has pieces given by flip sequences $\gamma$ with length $n$. 
We extend the notation $R_\gamma$ for the piece given by a finite flip sequence $\gamma$.
Let $P_\gamma$ to be its restriction to widths. 
Let $\xi_\gamma$ be the return time function for the piece $R_\gamma$. 

For the rest of the paper, we analyse the maps $P_\gamma$ and the return time function $\xi_\gamma$.

\subsection{Normalisation}
We now show that the sum of widths can be normalised to one without substantially changing Jacobians and return times.
This is makes it possible to invoke the theory of projective linear maps from the standard simplex to itself.

For $x \in \RR^{\calA}_{\geq 0}$, let $|x| \defeq \sum_{r \in \calA} x_r$.

\begin{lemma}\label{l:h-bounds}
For any balanced triangulation $\tau$ and any $(x, y)  \in \up_\tau$, we have
\[
3 < | x | < 6g- 6 + 3|Z|.
\]
\end{lemma}

\begin{proof}
For balanced triangulations $x_r \leqslant 1$ for all $r \in \calA$. 
The number of edges is $6g- 6 + 3 | Z| $. 
This gives the upper bound.

By the definition of $\up_\tau$, some forward flippable edge $e$ has width one. 
By considering widths of edges in $Q(e)$, the lower bound follows.
\end{proof}

Let $\Delta \subset \RR^{\calA}_{\geq 0}$ denote the standard simplex given by $| x | = 1$. 
Let $\rho\colon \RR^{\calA}_{\geq 0} \setminus \{0\} \to \Delta$ be the map given by $\rho(x) = x/ | x |$.
Let $A$ matrix with indices in $\calA$ and non-negative entries.
We define $\PP A \from \Delta \to \Delta$ as the map $\PP A (x) = \rho (A \, x)$. 

Let $\tau$ be a balanced triangulation. 
Recall from~\ref{s:balanced} that $\hor_\tau$ is the width polytope. 
Let $\pi_h: \calC^1 \to \hor_\tau$ be the projection $\pi_h(x,y) = x$. 
Let $\rho_\tau$ be the restriction of $\rho$ to the $\pi_h(\up_\tau)$. 
Let $\calJ \rho_\tau$ denote the Jacobian of $\rho_\tau$.
\begin{lemma}\label{l:rho-Jac}
There exists a constant $c > 1$ that depends only on the stratum component such that for any balanced triangulation $\tau$ and any $x \in \pi_h(\up_\tau)$
\[
\frac{1}{c} < \calJ \rho_\tau (x) < c.
\]
\end{lemma}

\begin{proof}
If $x \in \pi_h(\up_\tau)$ then it lies on a hyperplane given by $x_e = 1$ for some forward flippable label $e$. 
Hence, it suffices to show that $|x|$ is bounded above and below by constants that depend only on the stratum component.
This is exactly Lemma~\ref{l:h-bounds}.
\end{proof}

Let $H_\tau = \rho(\pi_h(\up_\tau))$.
Note that $H_\tau$ is a polytope in $\Delta$ of real dimension $D-1$. 
Let $\gamma$ be a finite flip sequence of labelled balanced triangulation from $\sigma = \tau^{(0)}$ to $\tau = \tau^{(n)}$.
Let $\PP A_\gamma | H_\tau$ be the restriction of $\PP A_\gamma$ to $H_\tau$.
Let $(x, y) = (x^{(n)}, y^{(n)})$ in $\up_\tau$ be the parameters that arise from initial parameters $(x^{(0)}, y^{(0)})$ in $\up_\sigma$.  
The map $P_\gamma$ can be expressed as the composition
\[
P_\gamma(x) =  \rho_\sigma^{-1} \circ \PP A_\gamma | H_\tau \circ \rho_\tau (x).
\]
In this composition only the map $\PP A_\gamma | H_\tau$ depends on $\gamma$. 
We will now estimate the Jacobian $\calJ P_\gamma$ of $P_\gamma$ in terms of the Jacobian of $\PP A_\gamma | H_\tau$.

\begin{lemma}\label{l:p-Jac}
There exists a constant $c> 1$ that depends only on the stratum component such that for any finite flip sequence $\gamma$ of labelled balanced triangulations from $\sigma = \tau^{(0)}$ to $\tau = \tau^{(n)}$ and any $x$ in the domain of $P_\gamma$, the Jacobian $\calJ P_\gamma$ satisfies
\[
\frac{1}{c} \calJ P_\gamma(x) < \calJ \PP A_\gamma | H_\tau (\rho_\tau(x) ) < c \calJ P_\gamma(x).
\]
\end{lemma}

\begin{proof}
By chain rule
\[
\calJ P_\gamma (x) = \calJ (\rho_\sigma)^{-1} (\PP A_\gamma|H_\tau (\rho_\tau(x)) \cdot  \calJ \PP A_\gamma | H_\tau (\rho_\tau(x)) \cdot \calJ \rho_\tau (x).
\]
The result follows from Lemma~\ref{l:rho-Jac}.
\end{proof}

We also estimate the return time $\xi_\gamma$.
\begin{lemma}\label{l:roof}
There exists a constant $c > 0$ that depends only on the stratum component such that for any finite flip sequence $\gamma$ of labelled balanced triangulations from  $\sigma = \tau^{(0)}$ to $\tau = \tau^{(n)}$ and any $x$ in the domain of $P_\gamma$ we have 
\[
\log | A_\gamma \rho_\tau (x) | - c < \xi(x) < \log | A_\gamma \rho_\tau (x) | + c.
\]
\end{lemma}
\begin{proof}
The return time for $\PP A_\gamma| H_\tau$ is $\log | A_\gamma \rho_\tau (x) |$. 
The result follows from Lemma~\ref{l:h-bounds}.
\end{proof}

By Lemmas~\ref{l:p-Jac} and~\ref{l:roof} we can now pass to the theory of projective linear maps.

\subsection{Jacobians of projective linear maps}

Given a balanced triangulation $\tau$ recall that $\pi_h \from \up_\tau \to \hor_\tau$ is the projection $\pi_h(x,y) = x$. 
Also recall $H_\tau = \rho(\pi_h(\up_\tau))$.

Let $\gamma$ be a finite flip sequence of labelled balanced triangulations from $\sigma = \tau^{(0)} $ to $\tau = \tau^{(n)}$
We will relate the Jacobian of the restriction $\PP A_\gamma | H_\tau$ to the Jacobian of $\PP A_\gamma$. 

We first analyse $\calJ \PP A_\gamma$. 
The expression for this is well known. 
See~\cite[Proposition 5.2]{Vee78}.
To keep the discussion self-contained, we include a proof.

Let $F \from \RR^{\calA}_{\geq 0} \setminus \{ 0 \} \times \RR^{\calA} \to \RR^{\calA}$ be the map
\[
F(z,w) = \frac{A_\gamma w}{| A_\gamma z |}.
\]
The map $\PP A_\gamma \from \Delta \to \Delta$ is the restriction to $\Delta$ of $u \mapsto F(u,u)$. 
We will compare the derivative of $u \mapsto F(u,u)$ to the derivative of $f_u(w) = F(u,w)$, that is we will compare the derivative along the diagonal subspace to the derivative along the vertical subspace.

We canonically identify the tangent vector in each factor with $\RR^{\calA}$.
Let $d = |\calA |$. 
Let $v \in \RR^{\calA}$.
At a point $(z,w)$ the total derivative $D \, F (0, v) = A_\gamma v / | A_\gamma z |$
Similarly,
\[
D \, F (v, 0) = d \left( \frac{1}{| A_\gamma z| }\right) (v) \cdot A_\gamma w = -\frac{\langle A_\gamma v, (1 \; 1 \cdots 1) \rangle}{| A_\gamma z |^2} \, A_\gamma w.
\]

For the composition $u \to F(u,u)$, consider the radial vector field $u \in T_u \RR^{\calA}$.
As a consistency check note that the image of the radial vector field under $u \mapsto F(u,u)$ is given by
\[
D \, F (u, u) = \frac{A_\gamma u}{| A_\gamma u|} - \frac{\langle A_\gamma u, (1 \; 1 \cdots 1) \rangle}{| A_\gamma u |^2} \, A_\gamma u = \frac{A_\gamma u}{| A_\gamma u|} - \frac{A_\gamma u}{| A_\gamma u|} = 0.
\]

We fix the splitting $T_u \RR^{\calA} = T_u \Delta \oplus \RR(u)$ and similarly $T_{\PP A_\gamma u} \RR^{\calA} = T_{\PP A_\gamma u} \Delta \oplus \RR (\PP A_\gamma u)$.
For $v \in T_u \Delta$, observe that
\[
D \, F(v,v) - D \, F(0,v) \in \RR(A_\gamma u).
\]
We deduce that the projections of $D \, F(v,v) $ and $D \, F(0,v)$ to $T_{\PP A_\gamma u} \Delta$
are identical.
Hence, we may compute $\calJ \PP A_\gamma$ by analysing the projection to $T_{\PP A_\gamma u} \Delta$ of $D \, F(0,v)$.

With our choices of splittings of $T_u \RR^{\calA}$ and $T_{\PP A_\gamma u} \RR^{\calA}$ the full derivative of the map $f_u\colon w \mapsto F(u,w) = A_\gamma w / |A_\gamma u|$ at the point $w = u$ has the matrix form
\begin{equation*}
Df_u \Bigr|_{w = u} = \left(
\begin{array}{ccc|c}
 & & & 0 \\
 & D \,  \PP A_\gamma 
 & & 0\\
 & & & \vdots \\
 & & & 0 \\ \hline
* & \cdots & * & 1 \\
\end{array}
\right).
\end{equation*}
The map $f_u\colon w \to F(u, w)$ is the composition $w \to A_\gamma w \to A_\gamma w/ | A_\gamma u |$.
Since $\det(A_\gamma) = 1$, the pullback satisfies
\begin{equation}\label{e:vol}
(A_\gamma)^\ast \left(d \, \vol_{\RR^{\calA}} \right) (u) = 	d \, \vol_{\RR^{\calA}}(A_\gamma u).
\end{equation}
The second map in the composition $w \mapsto A_\gamma w \mapsto A_\gamma w/ | A_\gamma u |$ is just a scaling.
So by~\ref{e:vol}
\begin{equation*}
(f_u)^\ast \left(d \, \vol_{\RR^{\calA}} \right) (u) = \frac{1}{| A_\gamma u |^d} \, d \, \vol_{\RR^{\calA}}(A_\gamma u).
\end{equation*}
This implies that the determinant of the above matrix up to sign is $1 / | A_\gamma u |^d$ from which we conclude that
\begin{equation}\label{e:full-Jac}
\calJ \PP A_\gamma (u ) = \frac{1}{| A_\gamma u |^d}.
\end{equation}

Recall that 
\[
D = \frac{1}{2} \dim_{\RR} \calQ(\kappa).
\]
In other words, $D-1$ is the real dimension of $H_\tau$ and $H_\sigma$.

\begin{lemma}\label{l:Jac}
Let $\gamma$ be a finite flip sequence of labelled balanced triangulations from $\sigma = \tau^{(0)} $ to $\tau = \tau^{(n)}$. There exists a constant $a_\gamma>0$ that depends only on $\gamma$ such that for any $u \in H_\tau$
\[
\calJ \PP A_\gamma | H_\tau (u) = \frac{1}{a_\gamma | A_\gamma u |^D}.
\]
\end{lemma}

\begin{proof}

Let $\RR H_\tau$ be the linear subspace generated by $H_\tau$.
We fix a subspace $V_\tau \subset \RR^{\calA}$ that is the orthogonal complement of $\RR H_\tau$ under the standard inner product.

For $u \in H_\tau$, let $D \, A_\gamma\colon T_u \RR^{\calA} \to T_{A_\gamma u} \RR^{\calA}$ be the map on the tangent space.
Note that $D\, A_\gamma (\RR H_\tau) = \RR H_\sigma$.
Now let $w$ be a vector in $V_\tau$.
Using the splitting $T_{A_\gamma u} \RR^{\calA} = \RR H_\sigma \oplus V_\sigma$, we write $D \, A_\gamma (w) = w_h + w_v$ where $w_h \in \RR H_\sigma$ and $w_v \in V_\sigma$.
Since $w_h \in \RR H_\sigma$, we can further decompose $w_h = w_H + c A_\gamma u$ where $w_H \in T_{A_\gamma u} \Delta$.
In other words, we have separated the radial part $A_\gamma u$ out explicitly.
In considering $D \, \PP A_\gamma\colon T_u \Delta \to T_{\PP A_\gamma u} \Delta$ the radial part gets ignored.
Hence, applying the scaling $1 / | A_\gamma u |$ in the final step note that for any $v \in T_u \Delta$ we have
\[
D \, \PP A_\gamma (v) = \frac{1}{ | A_\gamma u |} (w_H + w_v),
\]
where $w_v$ does not depend on $u$.
Hence, the matrix for $D \, \PP A_\gamma\colon T_u \Delta \to T_{\PP A_\gamma u} \Delta$ has the form:
\begin{equation*}
\left(
\begin{array}{ccc|ccc}
 & & & * & \cdots & * \\
 & D \, \PP A_\gamma | H_\tau  
 & & \vdots & \ddots & \vdots \\
 & & & * & \cdots & *\\ \hline
0 & \cdots & 0 & & & \\
\vdots & \ddots & \vdots & & M & \\
0 & \cdots & 0 & & &
\end{array}
\right),
\end{equation*}
where the bottom right $(d-D) \times (d-D)$ square block $M$ can be written as $| A_\gamma u | \cdot N$ where $N$ is a matrix that depends only on $A_\gamma$ and not on $u$.
Let $a_\gamma = \det(N)$.
Then
\[
\det(M) = \frac{a_\gamma}{| A_\gamma |^{d-D}}.
\]
Finally note that
\[
\calJ \PP A_\gamma (u) = \det(M) \cdot \calJ \PP A_\gamma | H_\tau (u).
\]
The lemma now follows from the expression of $\det(M)$ and~\ref{e:full-Jac}.
\end{proof}

\subsection{Combinatorial quirks}
\label{ss:quirks}

Compared to the interval exchange theory, veering triangulations throw up subtle issues that require care in the derivation of estimates.
For instance, we could call flip sequences in which every label flips at least once as complete flip sequences by analogy with interval exchange maps.
The hope would be that a sufficiently long concatenation of complete flip sequences has a positive matrix.
However, this is not true.
We show this by a simple example.
Consider the standard veering triangulation on a torus with three edges.
Let $c$ be the label of the wide edge with $a$ and $b$ the labels anti-clockwise from $c$. 
We will also assume that $c$ is blue.
Consider the following flip sequence $\gamma$ in the order
\begin{enumerate}
\item\label{flip1} flip $c$ such that $b$ is wide,
\item\label{flip2} flip $b$ such that $c$ is wide,
\item\label{flip3} flip $c$ such that $a$ is wide,
\item\label{flip4} flip $a$ such that $c$ is wide.
\end{enumerate}
Figure~\ref{fig:gamma} shows $\gamma$. 

\begin{figure}[ht]
\centering
\includegraphics{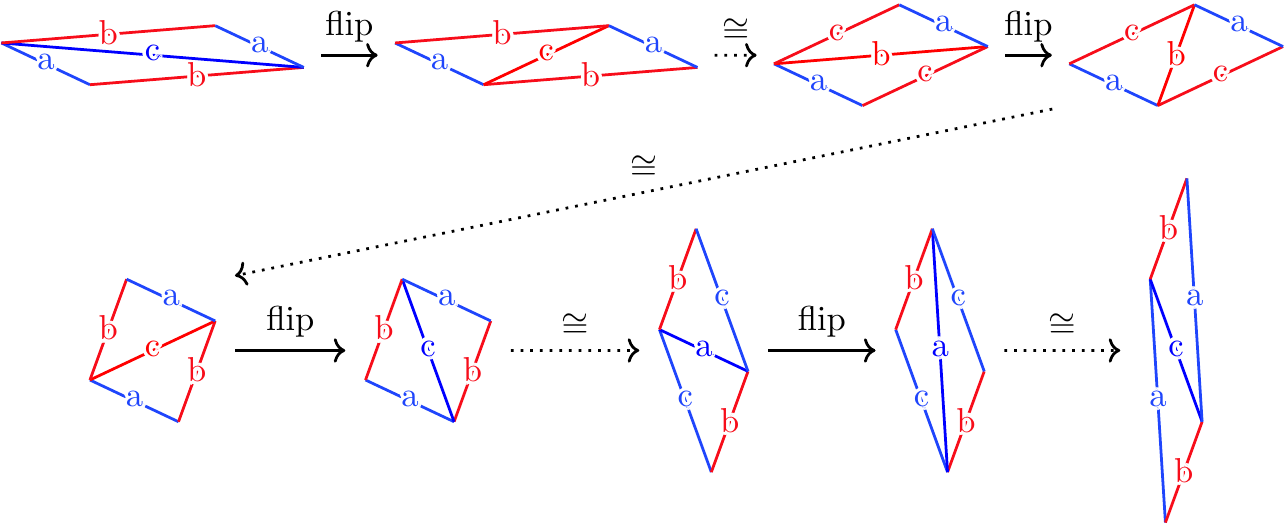}
\caption{The flip sequence corresponding to $\gamma$. In every sub-figure, edges with the same label are glued by translations.}
\label{fig:gamma}
\end{figure}

As a homeomorphism, this is the hyperbolic map $R^2 L^2$ where $R$ is the left Dehn twist in the $(0,1)$ curve and $L$ is the right Dehn twist in the $(1, 0)$ curve.
Clearly $\gamma$ is complete and returns us to the original triangulation.
The associated matrix is
\[
A_\gamma = \left( \begin{array}{ccc} 1 & 2 & 0 \\ 2 & 5 & 0 \\ 2 & 6 & 1\end{array} \right)
\]
and it is not primitive; the standard basis vector corresponding to $c$ is fixed, that is, $A_\gamma e_c = e_c$.

The example shows that it difficult to construct positive matrices by simple combinatorial criteria. 
It is reminiscent of the notion of infinitesimal branches in Bestvina--Handel theory for train tracks~\cite{Bes-Han}.
It is easily verified that under the concatenation
\[
\gamma^k = \underbrace{\gamma \gamma \cdots \gamma}_{k \; \text{times}}
\]
as $k \to \infty$ the heights $y^{(k)}_a, y^{(k)}_b \to \infty$ while $y^{(k)}_c = y_c$.
This means that, projectively, the branch of the vertical train track that crosses $c$ is infinitesimal in the sense of Bestvina--Handel.
To avoid a lengthy digression, we will not discuss these aspects of train track theory any further.

Similar examples also illustrate the subtleties in analysing distortion.
Consider the sequence $\gamma_n$ where flip~\eqref{flip2} in $\gamma$ is replaced by $n$ repetitions of the sequence  given by
\begin{enumerate}
\item flip $b$ such that $c$ is wide
\item flip $c$ such that $b$ is wide,
\end{enumerate}
before moving to flips~\eqref{flip3} and~\eqref{flip4}.
The sequence $\gamma_n$ corresponds to the homeomorphism $R^{2n} L^2$ which is hyperbolic.
The corresponding matrix is
\[
A_{\gamma_n} = \left( \begin{array}{ccc} 1 & 2 & 0 \\ 2n & 4n+1 & 0 \\ 2n & 4n+2 & 1\end{array} \right).
\]
Clearly, the matrix $A_{\gamma_n}$ is unbalanced.
But it is misleading to think that $\PP A_{\gamma_n}$ has large distortion.
The space of widths is cut out by $x_c = x_a + x_b$.
So it is the line segment between the vectors $(1/2, 0, 1/2)$ and $(0, 1/2, 1/2)$.
When restricted to this segment, $\PP A_\gamma$ has bounded distortion.
See Figure~\ref{fig:bounded_distortion}.
Note that for any point $x$ in the segment the norm $| A_{\gamma_n} \, x |$ is $O(n)$.
By Lemma~\ref{l:Jac} the Jacobian of $\PP A_\gamma$ is roughly the same at all points of the segment.
This shows that while a balanced matrix implies bounded distortion the converse need not hold; it is possible to have bounded distortion without a balanced matrix.
Hence we set up estimates for the underlying vertex cycles (extremal points of the space of widths) rather than the vertices of the ambient simplex.

\begin{figure}[ht]
\centering
\includegraphics{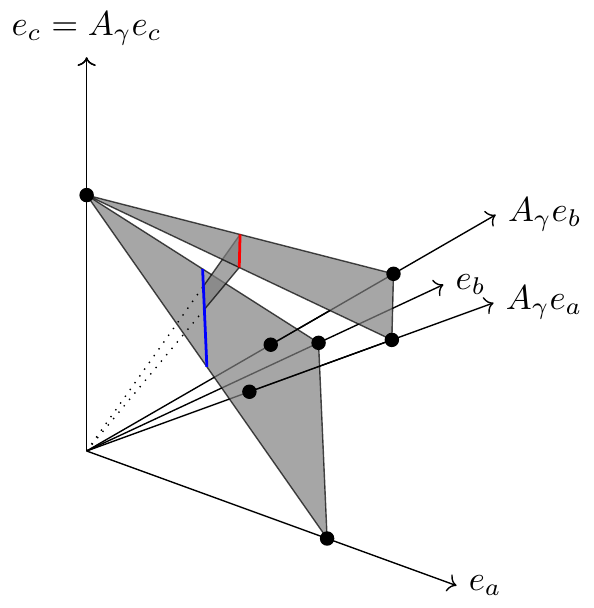}
\caption{An unbalanced matrix can have bounded distortion on a parameter subspace.}
\label{fig:bounded_distortion}
\end{figure}

In the interval exchange theory, a Rauzy sequence with a positive matrix is to construct a pre-compact transversal~\cite[Sections~4.1.2 and 4.1.3]{Avi-Gou-Yoc}.
We will construct the required transversal in an indirect manner without knowing whether the flip matrix is positive.

\subsection{Vertex cycles}

Let $\tau$ be a labelled balanced triangulation. 
The set $\hor_\tau \cap \Delta$ is convex with finitely many vertices.
Let $\cyc(\tau)$ be the set of its vertices.
A vertex $u$ can be written as a convex linear combination
\begin{equation}\label{e:full-convex}
u = \sum\limits_{s \in \calA} \alpha_{u, s } \, e_s
\end{equation}
of the standard basis vectors $e_s$ that form the vertex set of $\Delta$.
Because of the relations~\ref{eq:triangle-equality}, note that the $\alpha_{u,s} > 0$ for at least one forward flippable label $s$.
This implies that at least two coefficients in every convex linear combination are non-zero.
We say that a vertex $u$ \emph{contains} $s$ if $\alpha_{u,s} > 0$.
 
Let $B$ be the minimum of the non-zero coefficients $\alpha_{u,s}$ that may arise over all vertex cycles $u$ of all balanced triangulations $\tau$.
Note that $0 < B < 1/2$. 
We deduce that every coefficient $\alpha_{u,s}$ also satisfies $\alpha_{u,s} \leqslant 1-B$.
Also note that for every label $r$ there is some $u \in \cyc(\tau)$ that contains $r$.

\subsection{Kerckhoff lemma}

Suppose $\gamma$ is a finite flip sequence. 
Kerckhoff's strategy is to estimate from above the measure of the set of extensions of $\gamma$ that have bad distortion relative to the largest return time.
For interval exchanges, the return time is directly related to column norms of the associated matrix. 
For veering triangulations, we have to relate it to norms of vertex cycles which is harder.

For any finite sequence $\gamma$ of labelled balanced triangulations let $v_{\gamma,s}$ be the column labelled $s$ in $A_\gamma$, that is, $v_{\gamma, s} = A_\gamma e_s$ where $e_s$ is the standard basis vector. 

Suppose $\delta$ is a finite extension of $\gamma$.
We may write $\delta$ as a concatenation $\delta = \gamma \zeta$.
Suppose that a label $r$ is not flipped in $\zeta$.
Then $v_{\delta,r}$ has the form
\begin{equation}\label{e:column-form}
v_{\delta,r} = v_{\gamma,r} + \sum_{s \neq r} c_s v_{\gamma,s},
\end{equation}
where $c_s$ are non-negative integers.

Suppose $\gamma$ is a sequence from $\sigma$ to $\tau$, and $\zeta$ a sequence from $\tau$ to $\omega$. 
Let $H_\gamma = \PP A_\gamma (H_\tau) \subset H_\sigma$, $H_\delta = \PP A_\gamma (H_\tau) \subset H_\sigma$ and $H_\zeta = \PP A_\zeta (H_\omega) \subset H_\tau$. 
Let $\overline{x} \in H_\zeta$ and write
\[
\overline{x} = \sum_{s \in \calA} x_s e_s,
\]
the linear combination for $\overline{x}$ in terms of the standard basis vectors of $\Delta$.

\begin{lemma}\label{l:coefficients}
Suppose $\delta = \gamma \zeta$ is a finite extension of $\gamma$ and the label $r$ does not flip in $\zeta$. Suppose $\overline{x} \in H_\zeta$ has linear combination $\overline{x} = \sum_{s \in \calA} x_s e_s$.  Then for all $s \neq r$
\[
\frac{x_s}{x_r} \geq c_s.
\]
\end{lemma}

\begin{proof}
Let $x^{(\zeta)}_a$ denote the widths after $\zeta$.
Since $r$ does not flip, $x^{(\zeta)}_r = x_r$.
For any other $s \neq r$ we have $x_s = \sum\limits_{a \in \calA} A_\eta^{sa} x^{(\zeta)}_a$ where $A_\eta^{sa}$ denotes the $(s,a)^{\text{th}}$ entry of $A_\zeta$.
Note that $A_\zeta^{sr} = c_s$. 
Then $\sum\limits_{a \in \calA} A_\zeta^{sa} x^{(\zeta)}_a \geq A_\zeta^{sr} x_r = c_s x_r$ and we are done.
\end{proof}

Let $R > 1$.
Let $\Delta_R(r,s)$ be the subset whose convex linear combinations $\sum x_a e_a$ in $\Delta$ satisfying $x_s/x_r \geq R$.

\begin{lemma}\label{l:jump}
Let $\delta$ be a finite extension of $\gamma$. 
Let $M>1$ and suppose that $| v_{\delta,r} | > M | v_{\gamma,r} |$.
Then for some $s \neq r$ we have $\overline{x} \in \Delta_R(r,s)$ where
\[
R = \left( \frac{M-1}{d-1}\right) \frac{| v_{\gamma,r}| }{| v_{\gamma,s}|}.
\]
\end{lemma}

\begin{proof}
Using the form~\ref{e:column-form} for $v_{\delta,r}$ we get that
\[
\sum\limits_{s \neq r} c_s | v_{\gamma,s} | \geq (M-1) | v_{\gamma,r} |.
\]
This implies that
\[
c_s \geq \left( \frac{M-1}{d-1} \right) \frac{| v_{\gamma,r}| }{| v_{\gamma,s}|},
\]
for some $s \neq r$ and where $d = |\calA|$.
The result follows from Lemma~\ref{l:coefficients}.
\end{proof}

Consider $x \in H_\gamma$.
Note that all half-translation surfaces $(S, q(x,y))$ given by $x$ have the same forward flip sequence.
Suppose that $x$ gives vertically Keane half-translation surfaces.
By Theorem~\ref{t:staggered} the flip sequence of $x$ is infinite.
Fix $r \in \calA$.
In the infinite flip expansion of $x$, let $\gamma(x,r) = \gamma \zeta$ be the largest extension of $\gamma$ such that the label $r$ never flips in $\zeta$.
By Theorem~\ref{t:staggered} the sequence $\gamma(x,r)$ is finite for almost every $x$ in $H_\gamma$.

For a constant $M > 1$ we define
\[
X_{M,r} = \left\{ x \in H_\gamma \text{ such that } | v_{\gamma(x,r),r} | > M \max\limits_{s \in \calA} | v_{\gamma, s} | \right\}.
\]

Let $\ell$ be the Lebesgue measure on $H_\gamma$ by restriction of the standard Lebesgue measure on $\Delta$. 
We will now state a result which is often referred to as the Kerckhoff lemma~\cite[Lemma~A.1]{Avi-Gou-Yoc}.
We give an estimate similar to~\cite[Lemma~A.1]{Avi-Gou-Yoc}.

\begin{proposition}\label{p:K-lemma}
There exists $c> 0, M_0 > 1$ that depend only on stratum component such that for all $M>M_0$ and for any finite flip sequence $\gamma$ of labelled balanced triangulations 
\[
\ell(X_{M,r}) <  \frac{c}{M} \ell(H_\gamma).
\]
\end{proposition}

\begin{proof}
Suppose $\gamma$ is a sequence from $\sigma$ to $\tau$.
Let $\overline{X}_{M,r} = (\PP A_\gamma | H_\tau)^{-1} X_{M,r}$.
Notice that the condition $| v_{\gamma(x,r), r} | > M \max\limits_{a \in \calA} | v_{\gamma, a} |$ used to define $X_{M,r}$ is equivalently written as
\[
 | v_{\gamma(x,r), r} | > \left( \frac{M \, \max\limits_{a \in \calA} | v_{\gamma, a} | }{| v_{\gamma,r} |} \right) | v_{\gamma,r} |.
\]
Let $\overline{M}$ be the quantity inside the bracket on the right hand side above.
Comparing the above norm condition with the hypothesis in Lemma~\ref{l:jump}, we deduce that any $\overline{x} \in \overline{X}_{M,r}$ belongs to the set $\Delta_{\overline{R}}(r,s)$ for some $s$, where
\[
\overline{R} = \left( \frac{\overline{M}-1}{d-1}\right) \frac{| v_{\gamma,r}| }{| v_{\gamma,s}|}.
\]
Since
\[
\overline{M} - 1 \geq (M-1) \left( \frac{\max\limits_{a \in \calA} | v_{\gamma, a} |}{ | v_{\gamma, r} |} \right),
\]
we get $\overline{R} \geq R$, where
\[
R = \left( \frac{M-1}{d-1}\right) \frac{\max\limits_{a \in \calA} | v_{\gamma, a} |}{| v_{\gamma, s}|}.
\]
Since $\overline{R} \geq R$, we get the inclusion $\Delta_{\overline{R}}(r,s) \subset \Delta_R(r,s)$.
So any $\overline{x} \in \overline{X}_{M,r}$ also belongs to $\Delta_R(r,s)$.
Thus $X_{M,r} \subset \PP A_\gamma (\Delta_R(r,s))$.
Therefore, it is enough to prove that there exist $c>0, M_0 > 1$ that depend only on the stratum component such that 
\[
\ell(\PP A_\gamma ( \Delta_R(r,s) \cap \overline{X}_{M,r} )) <  \frac{c}{M} \ell( H_\gamma )
\]
for all $M > M_0$. 

Let $\cyc(\tau)$ be the set of vertices of $H_\tau$. 
We fix a vertex $\vartheta \in \cyc(\tau)$ that contains the edge labelled $r$ (refer to~\ref{e:full-convex}).
Then $\alpha_{\vartheta, r} \geq B$.
We cone the rest of the vertices in $\cyc(\tau)$ to $\vartheta$.
This triangulates $H_\tau$ by finitely many simplices $\Delta^i$ of dimension $D-1$.
Let $\cyc(\tau, i)$ be the vertex set for $\Delta^i$.
By construction, $\cyc(\tau,i) \subset \cyc(\tau)$ and $\vartheta \in \cyc(\tau,i)$.

The number of simplices $\Delta^i$ are bounded above by a constant that depends only on the stratum component. Hence, it suffices to show that the required estimate holds in each $\Delta^i$, that is, it suffices to show that there exist $c>0, M_0 > 1$ that depend only on the stratum component such that for all $M > M_0$ 
\begin{equation}\label{e:delta}
  \ell \left(\PP A_\gamma ( \Delta_R(r,s) \cap \Delta^i) \right) <   \frac{c}{M} \ell (\PP A_\gamma(\Delta^i))
\end{equation}
for all $\Delta^i$.

So consider a simplex $\Delta^i$.
Let $W_i \subset \cyc(\tau,i)$ denote the (possibly empty) subset of vertices $u$ of $\Delta^i$ that do not contain the edge labelled $r$.
If $W_i$ is empty then for any $z$ in $\Delta^i$ its convex linear combination
\[
z = \sum\limits_{s \in \calA} x_s \, e_s
\]
in terms of the standard basis vectors has to satisfy the constraint
\[
\frac{x_s}{x_r} \leqslant \frac{1-B}{B}.
\]
Note that by our hypothesis
\[
R = \left( \frac{M-1}{d-1} \right) \frac{\max\limits_{a \in \calA} | v_{\gamma, a} |}{| v_{\gamma, s} |} \geq \frac{M-1}{d-1}.
\]
This means that if $W_i$ is empty and $M$ is large enough so that $R > (1-B)/B$ then
the intersections $\Delta_R(r,s) \cap \Delta^i $ are empty for all $s \neq r$.
In that case, we have nothing to prove as the estimate~\ref{e:delta} is trivially satisfied in $\Delta^i$.

So we may assume that $W_i$ is non-empty.
As a convex linear combination of the vertices $u$ of $\Delta^i$, let us write a point $z$ in $\Delta^i$  as
\[
z = \sum\limits_{u \in \cyc(\tau,i)} b_u \, u.
\]
Relating the coefficients, we get
\[
x_a = \sum\limits_{u \in \cyc(\tau,i)} b_u \alpha_{u,a}
\]
for all $a \in \calA$.
Suppose now that $z$ belongs to $\Delta_R(r, s) \cap \Delta^i $.
Then $x_s \geq R x_r$ which gives
\[
\sum\limits_{u \in \cyc(\tau,i)} b_u \alpha_{u,s} \geq R \sum\limits_{u \in \cyc(\tau,i)} b_u \alpha_{u,r}.
\]
The above estimate is equivalent to
\[
\sum\limits_{u \in W_i} b_u \alpha_{u,s} \geq b_\vartheta (R \alpha_{\vartheta, r} - \alpha_{\vartheta, s}) + \sum\limits_{u \notin W_i, u \neq \vartheta} b_u (R \alpha_{u,r} - \alpha_{u,s}) + R \sum\limits_{u \in W_i} b_u \alpha_{u, r}.
\]
The last term on the right is non-negative. 
We may ignore it while maintaining the inequality to get
\[
\sum\limits_{u \in W_i} b_u \alpha_{u,s} \geq b_\vartheta (R \alpha_{\vartheta, r} - \alpha_{\vartheta, s}) + \sum\limits_{u \notin W_i, u \neq u_\vartheta} b_u (R \alpha_{u,r} - \alpha_{u,s}).
\]
Recall that we chose $M$ large enough such that $R > (1-B)/B$.
Hence $R \alpha_{u,r} - \alpha_{u,s} > 0$.
This makes the second term on the right hand side above positive. 
Thus it can be ignored while maintaining the inequality to get
\[
\sum\limits_{u \in W_i} b_u \alpha_{u,s} \geq b_\vartheta (R \alpha_{\vartheta, r} - \alpha_{\vartheta, s}).
\]
In fact, we will choose $M$ even larger such that $R > 2(1-B)/B$. 
This gives the estimate
\[
\sum\limits_{u \in W_i} b_u \alpha_{u,s} > \left( R - \frac{1-B}{B} \right) b_\vartheta \alpha_{\vartheta, r} = \frac{R}{2} \, b_\vartheta \alpha_{\vartheta, r}.
\]
Since $W_i$ is a subset of vertices of the $(D-1)$-dimensional simplex $\Delta^i$ its cardinality is at most $D$.
In particular, for some $u \in W_i$ we have
\[
b_u \alpha_{u,s} > \frac{R}{2D} \, b_\vartheta \alpha_{\vartheta, r}.
\]
The above inequality and the bound on the $\alpha$'s implies that for this $u \in W_i$ 
\begin{equation}\label{e:wedge}
b_u > \frac{RB}{2D(1-B)} \, b_\vartheta.
\end{equation}
Let
\[
R' = \frac{RB}{2D(1-B)}
\]
and let $\Delta^i_{R'}(\vartheta, u)$ denote the subset of points in $\Delta^i$ whose convex linear combination (in terms of the vertices of $\Delta^i$) satisfies~\ref{e:wedge}.
Thus we have proved above that
\[
\Delta^i \cap \Delta_R(r, s) \subset \bigcup\limits_{u \in W_i} \Delta^i_{R'}(\vartheta,u).
\]
So we have to reduced our problem to bounding from above the ratio
\[
\frac{\ell\left( \PP A_\gamma (\Delta^i_{R'} (\vartheta, u)\right)}{\ell ( \PP A_\gamma (\Delta^i))}.
\]
By definition, this ratio is given by
\[
\frac{\ell\left(\PP A_\gamma (\Delta^i_{R'}(\vartheta ,u ) \right)}{\ell (\PP A_\gamma (\Delta^i))}
= \frac{\displaystyle\int\limits_{\Delta^i_{R'}(\vartheta,u)} \calJ  \PP A_\gamma | H_\tau (z) \, d\ell }{\displaystyle\int\limits_{\Delta^i} \calJ  \PP A_\gamma | H_\tau  (z) \, d\ell}.
\]
By Lemma~\ref{l:Jac}, 
\[
\frac{\displaystyle\int\limits_{\Delta^i_{R'}(\vartheta,u)} \calJ  \PP A_\gamma | H_\tau (z) \, d\ell }{\displaystyle\int\limits_{\Delta^i} \calJ  \PP A_\gamma | H_\tau  (z) \, d\ell} =  \frac{\displaystyle\int\limits_{\Delta^i_{R'}(\vartheta,u)} \frac{1}{a_\gamma | A_\gamma z|^D}\, d\ell }{\displaystyle\int\limits_{\Delta^i} \frac{1}{a_\gamma | A_\gamma z|^D }\, d\ell}
= \frac{\displaystyle\int\limits_{\Delta^i_{R'}(\vartheta,u)} \frac{1}{| A_\gamma z|^D}\, d\ell }{\displaystyle\int\limits_{\Delta^i} \frac{1}{| A_\gamma z|^D }\, d\ell}.
\]
We see that this amounts to simply using the projective linear theory in dimension $D$.
Recalling the classical theory (see for example~\cite[Corollary~9.3]{Gad}),
\[
\frac{\displaystyle\int\limits_{\Delta^i_{R'}(\vartheta,u)} \frac{1}{| A_\gamma z|^D}\, d\ell }{\displaystyle\int\limits_{\Delta^i} \frac{1}{| A_\gamma z|^D }\, d\ell} = \frac{| A_\gamma \vartheta |}{| A_\gamma \vartheta | + R' | A_\gamma u |}.
\]
Since the vertex cycle $u$ contains the label $s$, we have
\[
| A_\gamma u | \geq B | v_{\gamma,s} |.
\]
Using the expression for $R'$ and the bound derived above, we get
\[
R' | A_\gamma u | \geq \frac{B(M-1) \max\limits_{a \in \calA} | v_{\gamma, a} |}{2D(1-B)(d-1) | v_{\gamma, s} | } B | v_{\gamma, s} | = \frac{B^2(M-1)}{2D (1-B)(d-1)} \max\limits_{a \in \calA} | v_{\gamma, a} |.
\]
Combining the two bounds above we get the upper bound
\begin{eqnarray*}
\frac{\ell\left( \PP A_\gamma (\Delta^i_{R'} (\kappa, u)\right)}{\ell ( \PP A_\gamma (\Delta^i))} & = & \frac{| A_\gamma u_\gamma |}{| A_\gamma u_\gamma | + R' | A_\gamma u |}\\
&=& 1- \frac{R' | A_\gamma u |}{| A_\gamma u_\gamma | + R' | A_\gamma u |}\\
&<& 1- \frac{R' | A_\gamma u |}{\max\limits_{a \in \calA} | v_{\gamma, a} | + R' | A_\gamma u |}\\
&=& \frac{\max\limits_{a \in \calA} | v_{\gamma, a} |}{\max\limits_{a \in \calA} | v_{\gamma, a} | + R' | A_\gamma u |}\\
& < & \frac{2D(1-B)(d-1)}{2D(1-B)(d-1) + B^2(M-1)}.
\end{eqnarray*}

Since the constants $D, B$ and $d$ only depend on the stratum component, the required bound is satisfied for each $\Delta^i$, and we are done.
\end{proof}

As a consequence of Proposition~\ref{p:K-lemma} we derive the following fact which we will be of use in the iterative step in the proposition that follows.
For a subset $\calB \subset \calA$ let
\[
W_{M, \calB} = \left\{ x \in H_\gamma \text{ such that } | v_{\gamma(x,r),r} | \leqslant M \max\limits_{s \in \calA} | v_{\gamma, s} | \text{ for all } r \in \calB\right\}.
\]
\begin{corollary}\label{c:def-prob}
There exists $M_0 > 1$ that depends only on the stratum component such that for any subset $\calB$ and any $\epsilon > 0$ 
\[
\ell(W_{M, \calB}) > (1 - \epsilon) \ell(H_\gamma)
\]
for all $M > M_0$. 
\end{corollary}
\begin{proof}
Let $X_{M,r}^c$ denote the complement $H_\gamma \setminus X_{M,r}$.
Note that
\[
W_{M, \calB} = \bigcap\limits_{r \in \calB} X_{M,r}^c.
\]
Using Proposition~\ref{p:K-lemma} and the union bound, we deduce
\[
\frac{\ell(W_{M, \calB})}{\ell(H_\gamma)} \geq 1 - \sum\limits_{r \in \calB} \frac{\ell(X_{M,r})}{\ell(H_\gamma)} > 1 - d \,  \frac{c}{M}.
\]
If $M$ is large enough then the right hand side is greater than $1-\epsilon$ and we are done.
\end{proof}

\subsection{Distortion results}

As before let $\gamma$ be a finite flip sequence of labelled balanced triangulations from $\sigma = \tau^{(0)}$ to $\tau = \tau^{(n)}$. 
Let $C > 1$.
A vertex cycle $u \in \cyc(\tau)$ is said to be \emph{$C$-strong} if
\[
| A_\gamma u | > \frac{1}{C} \max\limits_{a \in \calA} | v_{\gamma, a} |.
\]
We say that the label $r$ is \emph{$C$-strong in $\gamma$} if every vertex cycle $u \in \cyc(\tau)$ that contains the edge labelled $r$ is $C$-strong.
Let $\calA_{\gamma, C}$ be the set of labels $r$ that are $C$-strong in $\gamma$.
 
\begin{definition}\label{d:distortion}
Let $K>1$. We say that a finite flip sequence $\gamma$ of labelled balanced triangulations from $\sigma = \tau^{(0)}$ to $\tau = \tau^{(n)}$ has $K$-bounded distortion if, for any pair of points $x, x'$ in $H_\tau$,
\[
\frac{1}{K}  < \frac{\calJ \PP A_\gamma | H_\tau (x)}{\calJ \PP A_\gamma | H_\tau (x') } < K.
\]
\end{definition} 

\begin{lemma}\label{l:strong-distortion}
If $\calA_{\gamma, C} = \calA$ then $\gamma$ has $K = C^D$-bounded distortion.
\end{lemma}

\begin{proof}
Any point $x$ in $H_\tau$ is a convex linear combination of some of its vertex cycles. Since $\calA_{\gamma, C} = \calA$ we deduce that
\[
| A_\gamma x | > \frac{1}{C} \max\limits_{a \in \calA} | v_{\gamma, a} |.
\]
The lemma follows immediately from~\ref{l:Jac}.
\end{proof} 

As before let $H_\gamma = \PP A_\gamma (H_\tau)$ and if if $\delta = \gamma \zeta$ is an extension of $\gamma$ from $\sigma$ to $\omega$ then let $H_\delta = \PP A_\delta (H_\omega)$ and $H_\zeta = \PP A_\zeta (H_\omega)$.
Note that $H_\delta \subset H_\gamma \subset H_\sigma$ and $H_\zeta \subset H_\tau$.

\begin{lemma}\label{l:rel-prob}
Suppose that a finite flip sequence $\gamma$ has $K$-bounded distortion. Let $\delta = \gamma \zeta$ be a finite extension from $\sigma$ to $\omega$. Then
\[
\frac{1}{K} \cdot \frac{\ell(H_\zeta)}{\ell(H_\tau)}  < \frac{\ell(H_\delta)}{\ell(H_\gamma)} < K \cdot \frac{ \ell(H_\zeta)}{ \ell(H_\tau)}.
\]
\end{lemma}

\begin{proof}
Using the definition of being $K$-distorted, the lemma follows by integrating Jacobians.
\end{proof}

By strong connectivity lemma~\ref{l:strong-connect}, there is a flip sequence from $\tau$ to $\sigma$. 
We take the smallest such sequence $\zeta^{(\tau, \sigma)}$. Note that $\zeta^{(\tau, \sigma)}$ has length less than the number of balanced triangulations. 
This implies that there exists a constant $M'$ that depends only on the stratum component such that $\max_{a \in \calA} | v_{\zeta^{(\tau, \sigma)}, a} | < M'$ for all $\tau$. 

The following lemma is immediate so we omit the proof.
\begin{lemma}\label{l:back-to-sigma}
Let $\gamma$ be any finite flip sequence that gives $\tau$. Then 
\[
\max_{a \in \calA} | v_{\gamma \zeta^{(\tau,\sigma)} , a} | < M' \max_{a \in \calA} | v_{\gamma, a} |.
\]
\end{lemma}

By lemma~\ref{l:back-to-sigma}, we may extend any sequence by the appropriate $\zeta^{(\tau, \sigma)}$ to get back to $\sigma$, that is, trace a loop in the core graph.
The price to pay for the extension is a bounded increase in the largest column norm.

Let $M_0>1$ be the constant required for Corollary~\ref{c:def-prob} to hold for any subset $\calB$ of labels and a fixed $1 > \epsilon > 0$. 
For the rest of the paper, we fix this value of $M_0$. 
For notational brevity, we rename the constant $M_0M'$ simply to $M$. 

\begin{definition}
Let $\gamma$ be a finite flip sequence from $\sigma$. Suppose that $\calA_{\gamma, C} \neq \calA$. A finite extension $\delta$ of $\gamma$ is \emph{strengthening} with constants $C',C > 1$ if 
\begin{itemize}
\item $\delta$ ends at $\sigma$, 
\item $\calA_{\delta, C'}$ strictly contains $\calA_{\gamma, C}$, and 
\item $ \max_{a \in \calA} | v_{\delta, a} | < M \max_{a \in \calA} | v_{\gamma, a} |$.
\end{itemize} 
\end{definition} 
We denote the set of strengthening extensions by $\Ext(\gamma, C', C)$.
The set $\Ext(\gamma, C', C)$ is finite. 
The quantity $2^{M \max_{a \in \calA} | v_{\gamma, a} |}$ is a crude upper bound for its cardinality.

The next proposition sets up the inductive step which shows that with a definite probability and with control on the maximum norm the number of strong labels goes up.

\begin{proposition}\label{p:strengthening}
There exist $0 < p < 1$ that depends only on the stratum component such that, for any constant $C> 1$ and any finite flip sequence $\gamma$ with $\calA_{\gamma, C} \neq \calA$, there exists a constant $C' > 1$ with
\[
\sum\limits_{\delta \in \Ext(\gamma, C', C)}  \ell (H_\delta)  > p \, \ell (H_\gamma).
\]
\end{proposition} 

\begin{proof}
We call a label that is not strong \emph{weak}.
Let $T$ be a triangle in the triangulation that arises from $\gamma$. 
Every vertex cycle that contains some edge in $T$ must contain the wide edge in $T$.
This means that if the label of the wide edge in $T$ is strong then all labels in $T$ are strong.
Similarly if labels of both non-wide edges in $T$ are strong then the wide label in $T$ is strong.

In a forward flip, the largest norm of columns in a flip matrix at most doubles. 
We deduce that if forward flippable label is $C$-strong then it is $2C$-strong after it flips.
This shows that so long as the norms in an extension are controlled the number of strong labels is non-decreasing. 

Since $S$ is connected and $\calA_{\gamma, C} \neq \calA$, there are triangles in $\tau$ that contain a strong and a weak label.
It follows that any such triangle $T$ has exactly one strong label and that strong label is not wide.
We denote the collection of such triangles as $\{T_1, T_2, \ldots, T_j\}$. 
Let $\{b_1, b_2, \ldots, b_j\}$ denote the corresponding strong labels.

Let $p = 1 - \epsilon > 0$. By the choice of $M_0$, the subset $W_{M_0, \calA}$ of $H_\gamma$ where all labels have to flip before their column norms exceed $M_0 \max\limits_{a \in \calA} | v_{\gamma,a} |$ has measure at least $p \ell(H_\gamma)$.
Pick a typical point in $W_{M_0, \calA}$ and consider the infinite extension that it determines.
Let $\delta$ be the smallest finite extension in this infinite sequence such that some $b_i$ becomes wide in the new triangle that is on the same side as $T_i$.
Breaking symmetry, we may assume that the label that becomes wide is $b_1$.
Let $\zeta$ denote the last flip in $\delta$.
Then $\delta$ is a concatenation $\delta = \eta \zeta$.
Let $\omega$ be the triangulation that $\eta$ gives.
Consider the triangle $T'_1$ in $\omega$ on the same side of $b_1$ as $T_1$.
Note that none of the labels in $T_1$ flip in $\eta$.
Hence, the labels of $T'_1$ are the same as $T_1$. 
Let $c_1$ be the wide label.
If $c_1$ is $CM_0$-strong in $\eta$ then the number of strong labels in $\eta$ is larger.
If $\eta$ ends at $\tau$, then we extend $\eta$ by $\zeta^{(\tau, \sigma)}$.
By lemma~\ref{l:back-to-sigma}, the extension $\eta \zeta^{(\tau, \sigma)}$ is in $ \Ext(\gamma, CM, C)$, and we are done.

So suppose that $c_1$ is weak in $\eta$. 
Note that $c_1$ flips in $\zeta$ to make $b_1$ wide.
But then $c_1$ is $CM_0$-strong in $\delta$.
Thus, the number of strong labels in $\delta$ is larger.
If $\delta$ ends at $\tau$, then we extend $\delta$ by $\zeta^{(\tau, \sigma)}$.
By lemma~\ref{l:back-to-sigma}, the extension $\delta \zeta^{(\tau, \sigma)}$ is in $ \Ext(\gamma, CM, C)$, and again we are done.
\end{proof}

\begin{definition}
Let $\gamma$ be a finite flip sequence from $\sigma$ to $\sigma$, that is, $\gamma$ traces a loop in the core graph. An extension $\delta$ of $\gamma$ is \emph{strong} with a constant $C >1$ if 
\begin{itemize} 
\item $\delta$ ends at $\sigma$, 
\item $\calA_{\delta, C} = \calA$, and 
\item $\max_{a \in \calA} | v_{\delta, a} | < M^{6g-6+ 3|Z|}  \max_{a \in \calA} | v_{\gamma, a} |$.
\end{itemize}
\end{definition}
We denote the set of strong extensions by $\Ext(\gamma, C)$. 
Note that $\Ext(\gamma, C)$ is a finite set.
 
Proposition~\ref{p:strengthening} asserts that, with a definite probability, we can find strengthening extensions that increase the number of strong labels. 
Hence, by iteratively applying this proposition, we get strong extensions. 
We give the precise proposition below:

\begin{proposition}\label{p:strong}
There exist constants $C> 1$ and $0 < p < 1$ that depend only on the stratum component such that for any finite flip sequence $\gamma$,
\[
\sum\limits_{\delta \in \Ext(\gamma, C)} \ell(H_\delta) > p \, \ell(H_\sigma).
\]
\end{proposition} 

\begin{proof}
Let $C' > 1$ be any constant such that $\calA_{\gamma, C'} \neq \calA$. By Proposition~\ref{p:strengthening} there exists $C'' > 1$ and $0 < p < 1$ that depend only on the stratum component such that 
\[
\sum_{\delta \in \Ext(\gamma, C'', C')} \ell(H_\delta) > p \ell(H_\sigma).
\]
We replace $\gamma$ by such a $\delta$ and iterate the process. 
Since $|\calA| = 6g-6 + 3|Z|$, in at most $6g- 6 + 3|Z|$ iterates we get a strong extension and we are done.
\end{proof} 

By lemma~\ref{l:strong-distortion}, a great extension has $K$-bounded distortion for $K = C^D$.
For notational brevity, we rename the constant $M^{6g-6 + 3|Z|}$ simply to $M$.

\subsection{A special sequence} 

We now construct a finite flip sequence with certain special attributes that will be needed later.
Let $\sigma = \tau^{(0)}$ be a balanced triangulation. 
Let $x$ be a point in $\interior (H_\sigma)$ such that $x$ is Keane.
Thus the forward flip sequence $\delta$ of $x$ is infinite.
Let $\tau^{(n)}$ be the labelled balanced triangulation given by the finite prefix $\delta | n$.

The infinite intersection $\bigcap \PP A_{\delta | n}  (H_{\tau^{(n)}})$ can be identified with the set of invariant measures for the vertical foliation.
The argument is exactly as~\cite[Section~8.1]{Yoc1}. 

We consider the sequences $\delta | n$ as backward flip sequences from $\tau^{(n)}$. 
The flip matrix also relates the heights going backwards. 
Let $\pi_v$ be the projection of $\up_{\tau^{(n)}}$ to $\ver_{\tau^{(n)}}$.
Let $V_{\tau^{(n)}}$ be the image under $\rho \circ \pi_v$ of the heights in $\up_{\tau^{(n)}}$.
The infinite intersection $\bigcap \PP A_{\delta | n}^{-1} (V_{\tau^{(n)}} )$ can be identified, by an identical argument, with the set of invariant measures for the horizontal foliation. 

It is a standard fact that almost every point in $\up_\sigma$ gives a half-translation whose vertical and horizontal foliations are both uniquely ergodic \cite{Mas,Vee}. 
Hence 
\[
\{ x  \} = \bigcap\limits_{n = 1}^{\infty} \PP A_{\delta | n} (H_{\tau^{(n)}}) 
\]
and similarly for the vertical parameters. 

Let $\cyc(\sigma)$ be the vertex set of $H_\sigma$.
Let $u$ be a vertex in $\cyc(\sigma)$.
As we may recall from the proof of Proposition~\ref{p:K-lemma}, we can cone all other vertices in $\cyc(\sigma)$ to $u$ to get a partition of $H_\sigma$ into simplices of dimension $D-1$.
We denote this partition as $\calP_u$. 
For a simplex $\Delta^i$ in $\calP_u$ we let $\calF^i$ be its faces.
Consider the open set
\[
H_\sigma - \bigcup_{\substack{u \in \cyc(\sigma) \\ \Delta^i \in \calP_u \\ F \in \calF^i}} F
\]
of $H_\sigma$. Let $W$ be a component of the above open set.

We now choose $x \in W$ and choose an open neighbourhood $N_x$ of $x$ that is compactly contained in $W$. 
We do a similar process for the polytope of heights to get a point $y$ and an open neighbourhood $N_y$ of $y$. 
Since the sets nest down to $\{x\}$ and they are polytopes, we have $\PP A_{\delta | n} (H_{\tau^{(n)}})  \subset N_x$ for all $n$ large enough.
Similarly, we may assume that $\delta | n$ considered backwards nests inside $N_y$. 
We fix the smallest such $n$, say $n_0$ that satisfies both inclusions. 
By extending $\delta | n_0$ by some other finite flip sequence if necessary, we may also assume that we have a finite  sequence $\theta$ that satisfies
\begin{itemize}
\item $\theta$ ends at $\sigma$, that is, $\theta$ traces a loop in the core graph, and
\item if $\theta = \zeta \eta = \eta' \zeta$ for some $\zeta$ and $\eta$, then $\zeta= \theta$ and $\eta, \eta'$ are empty.
\end{itemize} 
We call sequences that satisfy the latter condition \emph{neat}, a term that we borrow from the theory of interval exchanges. 
See~\cite[Section 4.1.3]{Avi-Gou-Yoc} or \cite[Section 6.2]{Avi-Res}.

We will fix this special sequence $\theta$ for the remainder of the paper.
Let 
\begin{equation}\label{e:theta-max}
m = \max_{a \in \calA} | v_{\theta, a} |.
\end{equation}
Let $H_\theta = \PP A_\theta (H_\sigma)$.
By construction, for any $w \in \cyc(\sigma)$ there is some $\Delta^i \in \calP_w$ such that $H_\theta$ is contained in $\Delta^i$.
Let $\cyc(\sigma,i) \subset \cyc(\sigma)$ be the set of vertices of $\Delta^i$.
Any vertex $z$ of $H_\theta$ has a convex linear combination 
\[
z = \sum\limits_{u \in \cyc(\sigma,i)} b_u u.
\]
Since $H_\theta$ is compactly contained the coefficients are bounded away from zero and one.
So there is a constant $0 < B < 1/2$ that depends only on the stratum component such that $B < b_u < 1-B$. 

We will now analyse the effect on distortion of extending an arbitrary finite flip sequence $\gamma$ (ending at $\sigma)$ by our special sequence $\theta$. 
\begin{lemma}\label{l:theta-distortion}
Let $\gamma$ be a finite flip sequence from $\sigma$ to $\sigma$. The extension $\gamma \theta$ has $K$-bounded distortion for $K = (m/B)^D$. 
\end{lemma}

\begin{proof}
Let $w$ be a vertex of $\cyc(\sigma)$ such that $\max_{u \in \cyc(\sigma)} | A_\gamma u | = | A_\gamma w|$.
Let $\Delta^i \in \calP_w$ be the simplex that compactly contains $H_\theta$.
Let $v \in \cyc(\sigma)$ and $z = A_\theta v/ | A_\theta v |$. Then, $z$ is a vertex of $H_\theta$. 
Using the convex linear combination for $z$ we get
\[
A_{\gamma \theta} \, v = A_\gamma A_\theta v = |A_\theta v | A_\gamma z = |A_\theta v | \sum\limits_{u \in \cyc(\sigma,i)} b_u A_\gamma u.
\]
Note that $1 \leq | A_\theta v | \leq m$. 
So we bound from below the norm of  $A_{\gamma \theta} \, v $  (without attempting to be sharp) as 
\[
| A_{\gamma \theta} \, v | \geq  b_w \cdot |A_\theta v | \cdot  | A_\gamma w | \geq B \,  | A_\gamma w |.
\]
Similarly, by using that $\sum b_u =1$ and~(\ref{e:theta-max}), we derive the upper bound as
\[
| A_{\gamma \theta} \, v | \leq m | A_\gamma w | 
\]
Since any point $x$ in $H_\sigma$ is a convex linear combination of a subset of $\cyc(\sigma)$ we deduce that 
\[
B \,  | A_\gamma w | \leq | A_{\gamma \theta} \, x | \leq m | A_\gamma w |.
\]
By~\ref{l:Jac}, for distinct points $x$ and $x'$
\[
\left( \frac{B}{m} \right)^D \leq \frac{\calJ \PP A_{\gamma \theta} | H_\sigma (x) }{\calJ \PP A_{\gamma \theta} | H_\sigma (x')} \leq \left( \frac{m}{B} \right)^D
\]
and we are done.
\end{proof} 

We say that a flip sequence $\gamma$ \emph{does not contain} $\theta$ if $\gamma$ cannot be written as the concatenation $\gamma = \zeta \theta \eta$.
As an immediate consequence of Proposition~\ref{p:strong} and Lemma~\ref{l:theta-distortion}

\begin{proposition}\label{p:distortion}
There exists a constant $K> 1$ that depends only on the stratum component such that there is countable collection of sequences $\gamma$ from $\sigma$ to $\sigma$ such that
\begin{enumerate}
\item $\gamma$ is a concatenation $\gamma = \zeta \theta$ and no strict prefix of $\gamma$ contains $\theta$,
\item $\gamma$ has $K$-bounded distortion, and
\item the sets $H_\gamma$ are disjoint and their union has full measure in $H_\tau$. 
\end{enumerate} 
\end{proposition}

\begin{proof}
By Proposition~\ref{p:strong}, we can find finitely many flip sequences $\gamma \in \Ext(\emptyset, C)$. 
Note that all such $\gamma$ have $K$-bounded distortion for $K= C^D$. 
Also by Proposition~\ref{p:strong}, the union of $H_\gamma$ has measure at least $p \, \ell(H_\sigma)$.

If $\gamma$ does not contain $\theta$, then we extend $\gamma$ by $\theta$. 
Since $\theta$ is neat, no strict prefix of $\gamma \theta$ contains $\theta$.
By Lemma~\ref{l:theta-distortion}, the extension $\gamma \theta$ has $K$-bounded distortion for $K = (m/B)^D$. 
Let $p' = \ell(H_\theta) / \ell(H_\sigma)$. By Lemma~\ref{l:rel-prob}, $p' /K < \ell(H_{\gamma \theta}) /  \ell(H_\gamma)$. This implies that
\[
\frac{\ell(H_{\gamma \theta})}{\ell(H_\sigma)} > \left( \frac{p'}{K }\right) \cdot \frac{\ell(H_\gamma)}{\ell(H_\sigma)}.
\]

So suppose $\gamma$ contains $\theta$. 
Let $\eta$ be the shortest prefix of $\gamma$ that contains $\theta$. 
This means that $\eta = \zeta \theta$. 
By Lemma~\ref{l:theta-distortion}, $\eta$ has $K$-bounded distortion for $K = (m/B)^D$.
Also note, $\ell(H_\gamma) \leq \ell(H_\eta)$. 
If distinct sequences $\gamma$ and $\gamma'$ give the same truncation $\eta$ then note that $H_\gamma \cup H_{\gamma'} \subset H_\eta$. 
So $\ell(H_\gamma)  + \ell(H_{\gamma'}) \leq \ell(H_\eta)$. 

We call each new sequence we obtain by the above process a \emph{level-one} sequence.
In anticipation of what follows, we denote the finite collection of level-one sequences by $\Pi^{(1)}$. 
Let $W_1 = \bigcup_{\gamma \in \Pi^{(1)}}  H_\gamma$. 
We have proved above that $\ell(W_1) \geq (pp' /K) \cdot \ell(H_\sigma)$. 
By Proposition~\ref{p:strong} and~(\ref{e:theta-max}), we also note that, for any new sequence $\gamma \in \Pi^{(1)}$,
\[
\max_{a \in \calA} | v_{\gamma, a} | < m M.
\]

Considering the complement $H_\sigma - W_1$, we can find finitely many finite flip sequences $\{\zeta\}$ such that 
\begin{itemize}
\item $H_\zeta \subset H_\sigma - W_1$ for all $\zeta$,
\item the sets $H_\zeta$ are disjoint,
\item  $H_\sigma - W_1 = \bigcup_\zeta H_\zeta$, and
\item $\max_{a \in \calA} | v_{\zeta, a} | < 2mM$. 
\end{itemize} 
Again, if some $\zeta$ contains $\theta$ then we truncate to reduce to the previous step and we are done.
So we assume that no $\zeta$ contains $\theta$. 
Again by Proposition~\ref{p:strong}, for each such $\zeta$ we can find finitely many extensions $\eta \in \Ext(\zeta, C)$.
As in the first step, we extend or truncate these extensions $\eta$ to get a new extensions $\eta'$ of $\zeta$ such that
\begin{itemize} 
\item $\eta'$ are $K$-distorted for $K = (m/B)^D$, and
\item $\max_{a \in \calA} | v_{\eta', a} | < 2(mM)^2$.
\end{itemize}
We call each such sequence a \emph{level-two} sequence.
We denote the finite collection of level-two sequences by $\Pi^{(2)}$.
Let $W_2 = \bigcup_{\gamma \in \Pi^{(2)}} H_\gamma$. 
By the estimating the measure in the same way, we get $\ell(W_2) \geq (pp' /K) \cdot \ell(W_1^c)$.

Inductively, we can construct a finite collection $\Pi^{(n)}$ of \emph{level-$n$} sequences such that the union $W_n = \bigcup_{\gamma \in \Pi^{(n)}} H_\gamma$ is contained in $H_\sigma - \bigcup_{j =1}^{n-1} W_j$ and satisfies $\ell(W_n) \geq (pp' /K) \cdot \ell (H_\sigma - \bigcup_{j =1}^{n-1} W_j)$, and every $\gamma \in \Pi^{(n)}$ satisfies 
\begin{itemize} 
\item $\gamma$ is $K$-distorted for $K = (m/B)^D$, and
\item $\max_{a \in \calA} | v_{\zeta, a} | < 2^{n-1} (mM)^n$. 
\end{itemize} 
For notational brevity, we will rename the constant $pp'/K$ that depends only on the stratum component simply to $p$.
Let 
\[
W = \bigcup\limits_{j = 1}^\infty W_j.
\]
We will prove by induction that
\[
\ell\left( \bigcup\limits_{j = 1}^n W_j \right) \geq p + p (1-p) + p(1-p)^2 + \cdots + p(1-p)^{n-1} = 1 - (1-p)^n.
\]
As mentioned above, the base case $n = 1$ is true by Proposition~\ref{p:strong}.
Suppose that the lower bound is true for $n-1$.
Then note that
\begin{align*}
\ell\left( \bigcup\limits_{j = 1}^n W_j \right) &\geq \ell\left( \bigcup\limits_{j = 1}^{n-1} W_j \right) + p \left( 1- \ell\left( \bigcup\limits_{j = 1}^{n-1} W_j \right) \right)\\
&\geq p + (1-p) \sum\limits_{j = 1}^{n-1} p (1-p)^j = \sum\limits_{j = 1}^n p (1-p)^j,
\end{align*}
where the second inequality uses the induction hypothesis.
It follows that $\ell(W) = 1$ and we are done.
\end{proof}

For notational brevity, we rename the constant $2mM$ that depends only on the stratum component simply to $M$. 
After renaming, note that, for $\gamma \in \Pi^{(n)}$, 
\begin{equation}\label{e:norms}
\max \limits_{a \in \calA} | v_{\gamma, a} | < M^n
\end{equation}

As an immediate consequence of Proposition~\ref{p:distortion}, we show that almost every infinite flip sequence has bounded distortion infinitely often.
Before we give the precise statement, we set up some notation.
Let $\sigma = \tau^{(0)}$ be a labelled balanced triangulation.
Suppose $x \in H_\tau$ determines an infinite forward flip sequence $\delta$.
Let $\tau^{(i)}$ be the labelled balanced triangulations in $\delta$.
We denote the finite flip sequence in $\delta$ that gives $\tau^{(n)}$ by $\delta | n$. 

\begin{proposition}
There exists a constant $K > 1$ that depends only on the stratum component such that, for any labelled balanced triangulation $\sigma$ and $\ell$-almost every $x \in H_\sigma$, the infinite flip sequence $\delta$ of $x$ has infinitely many values of $n \in \NN$ for which the finite sequence $\delta | n$ has $K$-bounded distortion.
\end{proposition}

\begin{proof}
By Proposition~\ref{p:distortion}, for $\ell$-almost every $x \in H_\sigma$, the infinite flip sequence $\delta$ of $x$ has a value of $n \in \NN$ for which $\delta | n$ has $K$-bounded distortion. 
This establishes the base case.

The proposition is now proved by a further induction on the number $j$ of times infinite sequences have finite prefixes that have $K$-bounded distortion. 
The induction proceeds from a $j^{\text{th}}$-instance of bounded distortion to a full measure set of extensions that give $(j+1)^{\text{th}}$-instance of bounded distortion.
By repeating the proof of Proposition~\ref{p:distortion} such extensions exist.
Finally a countable intersection of full measure sets has full measure and we are done.
\end{proof}

We will now derive \emph{normality} of flip sequences: every finite flip sequence occurs infinitely often (with an asymptotic frequency) along almost every infinite flip sequence.
We state and prove it below:

\begin{proposition}\label{p:normality}
Let $\zeta$ be a finite flip sequence of balanced triangulations starting from $\tau$.
For $\ell$-almost every $x \in H_\sigma$, the infinite flip sequence $\delta$ of $x$ has infinitely many $n$ such that when expressed as a concatenation, $\delta | n+m = (\delta | n) \zeta$. Let $\NN_\zeta$ be the set of such $n$. Then up to a uniform multiplicative constant that depends only on the stratum component
\[
\frac{| \NN_\zeta \cap [0,N] | }{N }  \asymp \frac{\ell(H_\zeta )}{\ell(H_\sigma)}
\]
\end{proposition}

\begin{proof}
The proof is only a slight modification of the proof of Proposition~\ref{p:distortion}. 
By a version of lemma~\ref{l:back-to-sigma}, we may arrange that the level-one sequences end at $\tau$.

Consider the set where $\zeta$ follows the level-one sequences. By~\ref{l:rel-prob}, up to a uniform multiplicative constant, the relative measure of this set has the same proportion as $\ell(H_\zeta) / \ell(H_\tau)$.
We now consider the complement and pass to sequences that have $K$-bounded distortion and are one level up. 
We then the process. 
The resulting estimates are set up exactly as in Proposition~\ref{p:distortion}.
We leave the details to the reader.
\end{proof}

\subsection{Pre-compact transversal:}
We are now set up to construct the required transversal that is smaller than $\calT$ and has the desired dynamics. 
Recall the special sequence $\theta$ that we constructed earlier.
As our (pre-compact) transversal $\calT_\theta$, we set $\calT_\theta = \rho_\sigma^{-1} \pi_h^{-1} H_\theta \cap  \rho_\theta^{-1} \pi_v^{-1} V_\sigma$.
The transversal $\calT_\theta$ has positive measure in $\calT$ by construction.
By ergodicity of $\psi$, almost every flow trajectory that starts in $\calT_\theta$ returns to $\calT_\theta$ infinitely often.

\subsection{Countable shift} 
Recall the finite collections $\Pi^{(n)}$ of level-$n$ flip sequences constructed in Proposition~\ref{p:distortion}.
Let $\Pi= \bigcup_{n \in \NN} \Pi^{(n)}$. 
The set $\Pi$ is countable and we define it as the alphabet for our coding.

Let $\gamma$ be a sequence in $\Pi$. 
By construction $\gamma = \gamma' \theta$, where no strict prefix of $\gamma$ contains $\theta$. 
Hence the concatenation $\theta \gamma$ gives a first return map $R_{\theta \gamma}$ to $\calT_\theta$.
From Proposition~\ref{p:distortion}, the union $\bigcup_{\gamma \in \Pi} H_\gamma$ has full measure in $H_\sigma$.
So the union $\bigcup_{\gamma \in \Pi} H_{\theta \gamma}$ has full measure in $H_\theta$. 
Hence, the images of $R_{\theta \gamma}$ cover a full measure set in $\calT_\theta$. 
A typical infinite flow trajectory can now be coded by an infinite string of alphabets from $\Pi$ by its itinerary.

Since $\theta$ is a fixed finite sequence, we also deduce that the sequences $\theta \gamma$ have bounded distortion. 
Since the heights are related by the same flip matrix, $\theta \gamma$ has the same distortion on the heights.
Thus, our coding satisfies the bounded distortion property in Definition~\ref{d:bdp}.

\subsection{Return times:} 
Let $\gamma \in \Pi$.
Let $\xi_{\theta \gamma}$ be the return time function for the first return $R_{\theta \gamma}$. 
By Lemma~\ref{l:roof}, up to a uniform additive constant that depends only on the stratum component, $\xi_{\theta \gamma} (x)$ is equal to $\log | A_{\theta \gamma} \, x |$.

Suppose $\gamma \in \Pi^{(n)}$, that is, $\gamma$ is a level-$n$ sequence.
We deduce from~(\ref{e:norms}) and~\ref{e:theta-max} that 
\begin{equation}\label{e:norms2}
\max \limits_{a \in \calA} | v_{\theta \gamma, a} | < m \, M^n.
\end{equation}
Since $\theta \gamma$ has $K$-bounded distortion, this implies that for any $(x,y) \in \calT_\theta$
\[
| A_{\theta \gamma} \, x | < K^{1/D} m M^n.
\]
So, up to a uniform additive constant that depends only on the stratum component, $\xi_{\theta \gamma} (x)$ is bounded above by $n \log M$. 

Let 
\[
Y_n = \bigcup\limits_{\gamma \in \Pi^{(n)}} H_{\theta \gamma}.
\]
We apply the same inductive argument as in the proof of Proposition~\ref{p:distortion} to estimate $\ell(Y_n)$ from above.
We then get that, for a value of $0 < p < 1$ that depends only on the stratum component,
\[
\ell(Y_n) \leqslant 1 - \sum\limits_{j = 1}^{n-1} \ell(Y_j) \leqslant 1 - \sum\limits_{j = 0}^{n-2} p(1-p)^{j} = (1-p)^{n-1}
\]
for all $n \geq 2$.
By estimate~\eqref{e:norms2} and the above upper bound we get 
\[
\int e^{h \xi} \, d\ell < \sum\limits_{n = 1}^{\infty} M^{hn} \, \ell(Y_n) \leqslant M^h \sum\limits_{n = 1}^{\infty} \left(M^h(1-p) \right)^{n-1}. 
\]
If $M^h(1-p) < 1$, that is if $h$ is small enough, then the sum on the right hand side is finite. 
This proves that the roof function has \emph{exponential tails}.
We obtain yet another proof that the Masur--Veech volume of the stratum component $\calC$ is finite.
Furthermore, this concludes the desired verification of the dynamical properties of the coding finishing the proof of Theorem~\ref{t:main}. 

\sloppy\printbibliography

\end{document}